\documentclass[12pt,a4paper,american,reqno]{amsart}
\usepackage{amssymb}
\usepackage{amsmath}
\usepackage{graphicx}
\usepackage{enumerate}
\usepackage{physics}
\usepackage{amssymb}
\usepackage{amsthm}
\usepackage{epsf}
\usepackage{amsbsy,amsmath}
\usepackage{mathtools}
\usepackage{mathrsfs}
\usepackage{amsfonts}
\usepackage{eucal}
\usepackage{graphics,mathrsfs}
\usepackage{amsthm}
\usepackage{secdot}
\usepackage{esint}
\usepackage{bbm}
\usepackage{varwidth}
\usepackage{tasks}
\usepackage{cite}
\usepackage[colorlinks,pdfpagelabels,pdfstartview = FitH,bookmarksopen = true,bookmarksnumbered = true,linkcolor = blue,plainpages = false,hypertexnames = false,citecolor = red,pagebackref=false]{hyperref}

\theoremstyle{plain}
\newtheorem{theorem}{Theorem}[section]
\newtheorem{cor}[theorem]{Corollary}
\newtheorem{lemma}[theorem]{Lemma}

\theoremstyle{definition}
\newtheorem{definition}[theorem]{Definition}

\newtheorem{remark}[theorem]{Remark}

\allowdisplaybreaks

\usepackage{xcolor}

\long\def\symbolfootnote[#1]#2{\begingroup
\def\thefootnote{\fnsymbol{footnote}}\footnote[#1]{#2}\endgroup}

\numberwithin{equation}{section}
\begin{document}

\title[Weighted Eigenvalue problems for mixed operators] {A weighted eigenvalue problem for mixed local and nonlocal operators with potential}
\author{R. Lakshmi, Ratan Kr. Giri and Sekhar Ghosh}
\address[R. Lakshmi]{Department of Mathematics, National Institute of Technology Calicut, Kozhikode, Kerala, India - 673601}
\email{lakshmir1248@gmail.com / lakshmi\_p220223ma@nitc.ac.in}
\address[Ratan Kr. Giri]{Department of Mathematics, The LNM Institute of Information Technology, Jaipur, India - 302031
}
\email{: giri90ratan@gmail.com / ratan.giri@lnmiit.ac.in }
\address[ Sekhar Ghosh]{Department of Mathematics, National Institute of Technology Calicut, Kozhikode, Kerala, India - 673601}
\email{sekharghosh1234@gmail.com / sekharghosh@nitc.ac.in}


\begin{abstract} 
 We study an {\it indefinite weighted eigenvalue problem} for an operator of {\it mixed-type} (that includes both the classical {\it $p$-Laplacian} and the {\it fractional $p$-Laplacian}) in a bounded open subset $\Omega\subset \mathbb{R}^N \,(N\geq2)$ with {\it Lipschitz boundary} $\partial \Omega$, which is given by 
\begin{align*}
    -\Delta_p u + (-\Delta_p)^su+V(x)|u|^{p-2}u&=\lambda g(x)|u|^{p-2}u~\text{in}~\Omega,\\
    u&=0~\text{in}~\mathbb{R}^N\setminus\Omega,
\end{align*}
where $\lambda >0$ is a parameter, exponents $0<s<1<p<N$, and $V, g\in L^q(\Omega)$ for $q\in \left(\frac{N}{sp}, \infty\right)$ with $V\geq 0, g > 0$ a.e. in $\Omega$. Using the variational tools together with a {\it weak comparison} and {\it strong maximum principles}, we investigate the existence and uniqueness of {\it principal eigenvalue} and discuss its qualitative properties. Moreover, with the help of {\it Ljusternik-Schnirelman category theory},  it is proved that there exists a {\it nondecreasing sequence of positive eigenvalues} which goes to infinity. Further, we show that {\it the set of all positive eigenvalues is closed}, and {\it eigenfunctions} associated with every {\it positive eigenvalue} are bounded.

\end{abstract}

\maketitle
\section{Introduction}
Linear and nonlinear {\it eigenvalue problems} are a fundamental and fascinating area of research due to their far-reaching implications for understanding complex phenomena in various fields, such as in data science (image and data processing), deep learning (in deep neural networks(DNNs) to model connectivity among the network layers),  and also in nonlinear resonance, pattern formation, glaciology, shape optimization \& spectral theory, bifurcations \& phase transitions, population dynamics, nonlinear elasticity, fluid dynamics etc. to mention a few. For more details, one can refer to \cite{ETT15, B18, G18, GR03, BH17} and the relevant references therein. These problems involve studying eigenvalues and associated eigenfunctions for the different types of operators, like the classical {\it Laplace operator, $p$-Laplacian}, and also for its nonlocal version known as the {\it fractional $p$-Laplacian}. In the past decades, interest towards the study of nonlinear problems involving a {\it mixed operator} given by the sum of the classical $p$-Laplacian and the fractional $p$-Laplacian has grown a lot. For the relevant contributions involving this {\it mixed-type operator}, we refer to \cite{BDVV, BDVV2, BDVV3, BDVV6, BMV, BM23, DM22, DPLV3, GK22, GL23, LG24, SVWZ} and allied references therein. 
\par In this connection, by looking into various mathematical and physical aspects, the present paper is devoted to study a {\it weighted eigenvalue problem} involving the {\it mixed operator} over an open bounded subset $\Omega \subset \mathbb{R}^N$, ($N\geq 2$)  with {\it Lipschitz} boundary $\partial \Omega$, that is given by
\begin{align}\label{B}
\begin{split}
    -\Delta_p u + (-\Delta_p)^su+V(x)|u|^{p-2}u&=\lambda g(x)|u|^{p-2}u~\,\,\text{in}~\Omega,\\
    u&=0~\,\,\text{in}~\mathbb{R}^N\setminus\Omega,
    \end{split}
\end{align}
where $\lambda >0$ is a parameter and exponents $0<s<1<p<N$. Here, $\Delta_p u:= \text{div}(|\nabla u|^{p-2}\nabla u)$, denotes the $p$-Laplacian and the $(-\Delta_p)^su$ denotes the fractional $p$-Laplacian, defined as 
	$$(-\Delta_p)^su(x) := \mbox{P.V.}\int_{\mathbb{R}^N}\frac{|u(x)-u(y)|^{p-2}(u(x)-u(y))}{|x-y|^{N+sp}}d y,\,\,\forall\, x\in \mathbb{R}^N,$$
where P.V. refers to the {\it Cauchy principal value} of an integral. We denote 
$$\mathcal{L}_{s,p}(u) :=-\Delta_p u+ (-\Delta_p)^s u,$$
and call it as {\it mixed operator}. We assume that the potential $V: \Omega \rightarrow \mathbb{R}$, and the weight function $g: \Omega \rightarrow \mathbb{R}$ satisfies the following two conditions: 
$$ V, g\in L^q(\Omega)\, \text{ for } \,q\in \left(\frac{N}{sp}, \infty\right) \,\text{ and }\, V\geq 0, g>0 \,\text{ a.e. in } \Omega.$$

 \par The problem \eqref{B} is {\it half-linear} i.e., if $u$ is a solution of \eqref{B}, then for every $c\neq 0$, $cu$ is also a solution of \eqref{B} however, the sum of any two nonzero solutions is not a solution. So, we expect the existence of {\it positive solutions} and {\it positive eigenvalues} that would share similar properties, as for the local case when $\mathcal{L}_{s,p}(u)=-\Delta_p u$ only, and $V=0$ with $g=1$. In this case, the equation \eqref{B} becomes the well-known {\it Dirichlet eigenvalue problem} for $p$-Laplacian operator, and there is a vast literature devoted to this problem, and we refer \cite{GP87, L06, L90, AH95} to list a few. Also, the {\it eigenvalue problem} for the {\it non-local operator} when $\mathcal{L}_{s,p}(u)=(-\Delta_p)^su $ only and $V=0$ with $g=1$ has been studied by many authors, e.g. see \cite{BP16, DKP16, FP14, GKR24, LL14} and references therein.\\ 
 \par Before entering into the main goal and motivation behind considering the problem \eqref{B} in the present paper, let us first recall the existing literature on {\it weighted eigenvalue problems} dealing with $p$-Laplacian, fractional $p$-Laplacian, and mixed operators, separately.
 \begin{itemize}
 \item \textbf{The case of $p$-Laplacian:} For the case when $V=0$ and $\mathcal{L}_{s,p}(u)=-\Delta_p u$ only, the equation \eqref{B} reduce to  
 \begin{align}\label{P1}
\begin{split}
    -\Delta_p u &=\lambda g(x)|u|^{p-2}u~\,\,\text{in}~\Omega,\\
    u&=0~\,\,\text{on }\,\partial \Omega.
    \end{split}
\end{align}
 The weighted eigenvalue problem \eqref{P1} has been discussed under several hypotheses on the weight functions $g$, see \cite{AH95, A11, C01} and references therein. These articles discuss the existence of {\it positive eigenvalues} (see Definition \ref{evalue}) and give a characterization of the {\it principal eigenvalue}, which is simple, isolated and least among all positive eigenvalues. In addition, a monotonicity result for a sequence of positive eigenvalues has been obtained, which indicates the {\it existence of a non-decreasing sequence of positive eigenvalues that goes to infinity}. Also, the above results were further extended to the weighted eigenvalue problem involving $p$-Laplacian with a potential $V$, given by
  \begin{align}\label{P2}
\begin{split}
    -\Delta_p u +V(x)|u|^{p-2}u&=\lambda g(x)|u|^{p-2}u~\,\,\text{in}~\Omega,\\
    u&=0~\,\,\text{on }\,\partial \Omega.
    \end{split}
\end{align}
There are many research papers concerning the problem \eqref{P2}. One can refer to \cite{OT88, BD06, GP23} for different types of assumptions on $V$ and $g$. We especially mention here the work of Cuesta and Quoirin \cite{CQ09},  where the functions $V, g$ are assumed to be in $L^r(\Omega)$ for some $r>Np$ if $1<p\leq N$, and $ r=1$ if $p>N$.\\
\item \textbf{The case of fractional $p$-Laplacian:} The aforementioned results as in the local case are also obtained in \cite{I23} for the case of $V=0$ and $\mathcal{L}_{s,p}(u)=(-\Delta_p)^su$ only, under the assumption of $g$ that belongs $L^{\frac{N}{sp}}(\Omega)$ and $g^+\neq 0$ in $\Omega$. Recently, Asso et. al. \cite{ADLQ23} studied the existence and properties of {\it principal eigenvalue} \eqref{B} for the case of $\mathcal{L}_{s,p}(u)=(-\Delta_p)^su$ only,  when $V$ and $m$ are {\it indefinite sign-changing} functions and satisfying the following conditions: $V,g\in L^r(\Omega)$ with $ r\in (1,\infty)\cap (\frac{N}{sp},\infty)$ with $g^+\neq 0$.\\
\item \textbf{The case of mixed operator:} When $V=0$ and $g=1$, the eigenvalue problem \eqref{B} has been studied in \cite{DFR19}. The authors obtained the existence and properties of the {\it principal eigenvalue} as well as studied the limiting case as $p \rightarrow \infty$. For the case when $V\in L^\infty(\Omega)$ and $g=1$, the eigenvalue problem \eqref{B} has been considered by Biagi et. al. \cite[Proposition 5.1]{BMV}. In this work, the authors proved the existence of a {\it positive principal eigenvalue} that is unique, simple, and isolated. Further, it is also obtained that {\it principal eigenfunctions} do not change the signs in $\mathbb{R}^N$, whereas {\it eigenfunctions} associated with {\it non-principal positive eigenvalues} are {\it nodal}, i.e. sign-changing. 
\end{itemize}
\vspace{0.2cm}
\par \textbf{The main goal:} The main purpose of the present paper is to study the existence of eigenvalues and eigenfunctions and characterize the properties for the case when $V, g\in L^q(\Omega)$ for $q\in \left(\frac{N}{sp}, \infty\right)$ with $V\geq 0, g>0$ a.e. in $\Omega$. To be more precise, we discuss here the {\it positive eigenvalues} associated with the problem \eqref{B}. We establish a {\it weak comparison principle}(Theorem \ref{B-CP}) for the operator $\mathcal{L}_{s,p}(u)$ with potential $V\in L^q(\Omega)$, $q\in \left(\frac{N}{sp}, \infty\right)$ along with a generalized {\it strong maxmium principle} (Theorem \ref{B-smp}). Then using {\it comparison principles} together with the {\it Raleigh quotient}, we study the existence and qualitative behaviours of {\it principal eigenvalues}. We also derive the {\it monotonicity} results of {\it positive eigenvalues} (as mentioned in the local and nonlocal cases), nodal domains and obtain the {\it boundedness} of {\it eigenfunctions} associated with {\it positive eigenvalues}. To prove monotonicity and boundedness results, we use the {\it Lusternick-Schinrelman category theory} and {\it Moser-type iteration}, respectively. We summarize here the central results of this paper as follows: 
\begin{theorem}
Let $\Omega$ be a bounded domain in $\mathbb{R}^N$, ($N\geq 2$), with Lipschitz boundary $\partial\Omega$ and $0<s<1<p<N$ . Assume that $V, g\in L^q(\Omega)$ for $q\in \left(\frac{N}{sp}, \infty\right)$ with $V\geq 0, g>0$ a.e. in $\Omega$. Then
\begin{itemize}
\item[(a)] there exists principal eigenvalue $\lambda_1(\Omega)$ of \eqref{B} which is strictly positive and satisfies the following properties:
\begin{itemize}
    \item[(i)] $\lambda_1(\Omega)$ is simple and isolated;
    \item[(ii)] the eigenfunction associated with the principal eigenvalue $\lambda_1(\Omega)$ don't change sign in $\Omega$;
    \item[(iii)] the eigenfunctions associated with non-principal eigenvalue $\lambda(\Omega)>\lambda_1(\Omega)$ is nodal i.e., change signs.  
\end{itemize}
\item[(b)] the problem \eqref{B} admits a monotone sequence of positive eigenvalues which goes to infinity.
\item[(c)] the set of all positive eigenvalues to \eqref{B} is closed.
\item[(d)] the eigenfunctions associated with every positive eigenvalue to \eqref{B} are bounded in $\mathbb{R}^N$.
\end{itemize}
\end{theorem}

The rest of the paper is organized as follows. In Section $2$, we recall the classical {\it Sobolev} and {\it fractional Sobolev spaces}, and provide basic results related to these spaces. Then, in Section $3$, we derive {\it weak comparison principle} of solutions to the corresponding homogeneous problem of \eqref{B} and certain estimates. As a consequence, we prove the strong maximum principle for eigenfunctions associated with positive eigenvalues to \eqref{B}. Section $4$ is devoted to study the existence and qualitative properties of the {\it principal eigenvalue}. Finally, in Section $5$, we show the existence of infinitely many {\it positive eigenvalues} going to infinity, and discuss the {\it closedness of the set of all positive eigenvalues}. At last, we obtain the global {\it boundedness of eigenfunctions} associated with every {\it positive eigenvalue}. 

\section{Prerequisites and space setup}\label{s2} 
We begin with some standard notations, definitions and the results for classical {\it Sobolev} and {\it fractional Sobolev spaces} \cite{B11, DPV12, L23}, which will be used throughout the paper.
\par Let $\Omega$ be a nonempty open set in $\mathbb{R}^N$ with dimension $N\geq 2$. The classical Sobolev space $W^{1,p}(\Omega)$ is defined as
\begin{equation*}
    W^{1,p}(\Omega)=\{u \in L^p(\Omega):\frac{\partial u}{\partial x_i} \in L^p(\Omega) \ \text{for all} \ 1 \leq i \leq N \}.
\end{equation*} 
It is well known that the space $W^{1,p}(\Omega)$ is a reflexive, separable Banach space equipped with the norm: 
\begin{equation}\label{1-norm}
    \|u\|_{1,p}=\left(\|u\|_{L^p(\Omega)}^p+\|\nabla u\|_{L^p(\Omega)}^p\right)^{\frac{1}{p}}.
\end{equation}
We denote $W_0^{1,p}(\Omega)$ as the closure of the space $C_c^{\infty}(\Omega)$ of smooth functions with compact support with respect to the Sobolev norm defined as in \eqref{1-norm}. It is also a Banach space, and can be characterized as
\begin{equation*}
    W_0^{1,p}(\Omega)=\{u \in W^{1,p}(\Omega):u=0 ~\text{on}~\partial\Omega \}.
\end{equation*} 
For $\displaystyle{0<s<1<p<\infty}$, the fractional Sobolev space $W^{s,p}(\Omega)$ is defined as
\begin{equation*}
    W^{s,p}(\Omega)= \left\{u \in L^p(\Omega): \frac{| u(x)-u(y)|}{| x-y | ^{\frac{N}{p}+s}} \in L^p(\Omega \times \Omega)\right\}
\end{equation*}
 and is equipped with the norm $\|u\|_{s,p}=\|u\|_{L^p(\Omega)}+[u]_{W^{s,p}(\Omega)}$, where
        \begin{equation*}
            [u]_{W^{s,p}(\Omega)}=
                \left(\int_{\Omega} \int_{\Omega}\frac{|u(x)-u(y)|^p}{|x-y|^{N+sp}} dx dy\right)^{\frac{1}{p}}, \mbox{$1 \leq p < \infty$}. 
        \end{equation*}  
Likewise, in the classical Sobolev spaces, the space $W_0^{s,p}(\Omega)$ denotes the fractional Sobolev space with zero boundary values, and it is defined by
\begin{equation*}
    W_0^{s,p}(\Omega)=\{u \in W^{s,p}(\mathbb{R}^N):u=0 ~\text{in}~\mathbb{R}^N\setminus\Omega\}.
\end{equation*} 
Both $W^{s,p}(\Omega)$ and $W_0^{s,p}(\Omega)$ are reflexive, separable Banach spaces for all $0<s<1<p<\infty$. The following defines the tail of a function in the fractional Sobolev space $W^{s,p}(\mathbb{R}^N)$, which will be used to obtain certain estimates. 
\begin{definition}[See \cite{DKP16}]
    Let $w \in W^{s,p}(\mathbb{R}^N)$. The tail of $w$ with respect to the ball $B_r(x_0)$ is defined by
    \begin{equation}\label{B-tail}
        \operatorname{Tail}(w;x_0,r)=\left(r^{sp}\int_{\mathbb{R}^N \setminus B_r(x_0)} |w(x)|^{p-1}|x-x_0|^{-(N+sp)}dx \right)^{\frac{1}{p-1}}.
    \end{equation}
\end{definition}
\vspace{0.4cm}
\par Next, we enlist the following {\it Sobolev inequalities} and the {\it Rellich-Kondrachov} type embedding results for the classical Sobolev and the fractional Sobolev spaces. 
\begin{theorem}[See \cite{DD12}]\label{SE}
    Let $\Omega\subset\mathbb{R}^N$ be a bounded domain with Lipschitz boundary $\partial\Omega$ and $1<p<N$. Then we have the continuous embedding,
    \begin{equation*}
        W^{1,p}(\Omega) \hookrightarrow L^{q}(\Omega), \ 1 \leq q 
 \leq p^*=\frac{Np}{N-p}.
    \end{equation*}
    Moreover, for all $u \in W^{1,p}(\mathbb{R}^N)$ we have
    \begin{equation}\label{SEE}
        \left(\int_{\mathbb{R}^N} | u|^{p^*}dx\right)^{\frac{1}{p^*}} \leq  S\left(\int_{\mathbb{R}^N} | \nabla u|^p dx\right)^{\frac{1}{p}},
    \end{equation}
    where $S>0$ is the best Sobolev constant.
\end{theorem}

\begin{theorem}\label{RK}(See \cite[Theorem 2.80]{DD12})
    Let $\Omega \subset \mathbb{R}^N$ be a bounded domain with Lipschitz boundary $\partial\Omega$ and $1<p<N$. Then we have the following Compact embedding,
    \begin{equation*}
     W^{1,p}(\Omega) \hookrightarrow L^q(\Omega),  \ 1 \leq q <p^*.    
    \end{equation*}
\end{theorem}
\begin{theorem} \label{emb_1}(See \cite{DPV12, DD12})
Let $\Omega\subset \mathbb{R}^N$ be a bounded domain with Lipschitz boundary $\partial\Omega$ and $0 < s < 1 < p <\infty$. Then the Sobolev space $W^{1,p}(\Omega)$ is continuously embedded in the fractional Sobolev space $W^{s,p}(\Omega)$ i.e., there exists a constant $ C = C(N, p, s) > 0$ such that 
\begin{equation}
\|u\|_{s,p} \leq C \|u\|_{1,p} \,\,,\, \forall \, u\in W^{1,p}(\Omega).\label{emb_eq1}
\end{equation}
Also for $\Omega=\mathbb{R}^N$, 
$$W^{1,p}(\mathbb{R}^N)\xhookrightarrow{} W^{s,p}(\mathbb{R}^N).$$ 
\end{theorem}

\begin{theorem}\label{emb_2}(See \cite[Lemma 2.3]{GK22})
Let $\Omega\subset \mathbb{R}^N$ be a bounded domain with Lipschitz boundary $\partial\Omega$ and $0 < s < 1 < p <\infty$. Then there exists a constant $C = C(N, p, s, \Omega)>0$  such that 
\begin{equation}
\int_{\mathbb{R}^N}\int_{\mathbb{R}^N}\frac{| u(x)-u(y)|^p}{| x-y |^{N+ps}} dxdy \leq C \int_\Omega |\nabla u|^p dx\,,\,\,\, \forall \, u\in W^{1,p}_0(\Omega).\label{emb_eq2}
\end{equation}

\end{theorem}
\begin{theorem}\label{FSE}(See \cite[Theorem 7.6]{L23})
 Let $0<s<1<p<N$ and $p_s^*=\frac{Np}{N-ps}$. Then we have the following continuous embedding:
     \begin{equation*}
      W^{s,p}(\mathbb{R}^N) \hookrightarrow L^q(\mathbb{R}^N),\  \  p \leq q \leq p_s^*.
     \end{equation*}
    If $u$ vanishes at infinity, then there exists a best (Sobolev) constant $S_s>0$ such that
    \begin{equation}\label{FSEE}
         \| u\|_{L^{p_s^*}(\mathbb{R}^N)} \leq S_s\left(\int_{\mathbb{R}^N}\int_{\mathbb{R}^N}\frac{| u(x)-u(y)|^p}{| x-y |^{N+ps}} dxdy\right)^{\frac{1}{p}}.
    \end{equation}
\end{theorem}
\subsection{Solution space setup} From now onwards, we shall assume $\Omega\subset \mathbb{R}^N$ is a bounded domain with a {\it Lipschitz} boundary $\partial \Omega$, and $p > 1$, $s \in (0,1)$ are real numbers such that $0<s<1<p<N$. Since we are dealing with the mixed local and nonlocal operators, we consider the following function space
$$   X_0^{s,p}(\Omega) = \{ u \in W^{1,p}(\mathbb{R}^N): u|_{\Omega} \in W_0^{1,p}(\Omega), \ u =0 \text{ in } \mathbb{R}^N \setminus \Omega \}$$  
endowed with the norm 
 $$   \|u\|_{X_0^{s,p}(\Omega)}
    =\left(\int_{\Omega}| \nabla u |^p dx+\int_{\mathbb{R}^N}\int_{\mathbb{R}^N}\frac{| u(x)-u(y)|^p}{| x-y |^{N+ps}} dxdy\right)^\frac{1}{p}.$$
{\it The function space $X_0^{s,p}(\Omega)$ will be used as a solution space to our eigenvalue problem \eqref{B}}. Using Theorem \ref{FSE} together with the embedding \eqref{emb_eq1} and the Poincar\'{e} inequality, one can see that there exist a constant $C>0$ such that
    \begin{equation}\label{norm-equiv}
        \| \nabla u \|_{L^p(\Omega)} \leq  \|u\|_{X_0^{s,p}(\Omega)} \leq C \| \nabla u \|_{L^p(\Omega)} \text{ for all } u \in X_0^{s,p}(\Omega).
    \end{equation}
Thus, the norm $\|\cdot\|_{W_0^{1,p}(\Omega)}$ is equivalent to $\|\cdot\|_{X_0^{s,p}(\Omega)}$ on $X_0^{s,p}(\Omega)$. Moreover, using the Sobolev inequality \eqref{SEE}, we have the mixed-type Sobolev inequality:
    \begin{equation}\label{MSI}
        \|u\|_{L^{p^*}(\Omega)} =\|u\|_{L^{p^*}(\mathbb{R}^N)}\leq S \|\nabla u\|_{L^p(\mathbb{R}^N)} \leq S\|u\|_{X_0^{s,p}(\Omega)}, \,\forall\, u \in X_0^{s,p}(\Omega).
    \end{equation}
Since $\Omega$ is bounded, we can also apply the H\"older's inequality in \eqref{MSI} to obtain the following continuous embedding result:
    \begin{equation}\label{CONT MET}
        X_0^{s,p}(\Omega) \hookrightarrow L^q (\mathbb{R}^N), \,\,1\leq q\leq p^*.
    \end{equation} 
    In addition, in light of the Theorem \ref{RK}, the embedding
    \begin{equation}\label{CPT MET}
        X_0^{s,p}(\Omega) \hookrightarrow L^q(\Omega)
    \end{equation}
    is compact for $1\leq q<p^*$. We have the following result.
\begin{theorem} The space $X_0^{s,p}(\Omega)$ is a Banach space for all $1\leq p\leq \infty$. Moreover, it is reflexive for $1\leq p<\infty$ and separable for $1<p<\infty$, respectively.
\end{theorem}
\begin{proof}
    Let $(u_n)$ be a cauchy sequence in $X_0^{s,p}(\Omega)$. Since $\|\nabla u\|_{L^p(\Omega)} \leq \|u\|_{X_0^{s,p}(\Omega)}$ for every $u \in X_0^{s,p}(\Omega)$, we get $(\nabla u_n)$ is a cauchy sequence in $L^{p}(\Omega)$. Thus, there exists $w \in L^{p}(\Omega)$ such that $\nabla u_n \rightarrow w$ in $L^{p}(\Omega)$. For every $n \in \mathbb{N}$, we have $u_n \in W_0^{1,p}(\Omega)$. By the Poincar\'{e}'s inequality (see \cite{B11}, Corollary 9.19), $u \mapsto \|\nabla u\|_{L^p(\Omega)}$ gives a norm which is equivalent to $\|.\|_{W^{1,p}(\Omega)}$ in $W_0^{1,p}(\Omega)$. Since $W_0^{1,p}(\Omega)$ is a Banach space, there exists $u \in W_0^{1,p}(\Omega)$ such that $u_n \rightarrow u$ in $W_0^{1,p}(\Omega)$. Extending $u$ to $\mathbb{\mathbb{R}^N} \setminus \Omega$ by zero (still denoted by $u$), we get $u \in W^{1,p}(\mathbb{R}^N)$. Therefore, we obtain $u \in X_0^{s,p}(\Omega)$. On using $u_n \rightarrow u$ in $W_0^{1,p}(\Omega)$, we have 
    \begin{equation}\label{B-conv}
        \|\nabla u_n - \nabla u\|_{L^p(\Omega)} \rightarrow 0 \text{ as } n \rightarrow \infty.
    \end{equation}
    From the equations \eqref{norm-equiv} and \eqref{B-conv}, we get $u_n \rightarrow u$ in $X_0^{s,p}(\Omega)$. Therefore, $X_0^{s,p}(\Omega)$ is a Banach space for all $1\leq p<\infty$. The case $p=\infty$ can be proved following a similar argument. Now define  $T:X_0^{s,p}(\Omega) \rightarrow L^{p}(\Omega) \times L^p(\mathbb{R}^N \times \mathbb{R}^N)$ by $T(u)=(a_u,b_u)$, where $$ a_v(x)= \nabla v(x), \,\,\,
     b_u(x,y)= \frac{v(x)-v(y)}{| x-y| ^{\frac{N}{p}+s}}.$$ 
     Thus for every $u \in X_0^{s,p}(\Omega)$, we have
     $$\|Tu\|_{ L^p(\Omega) \times L^p(\mathbb{R}^N \times \mathbb{R}^N)}=\|u\|_{X_0^{s,p}(\Omega)}.$$ Therefore $T$ becomes an isometry that maps $X_0^{s,p}(\Omega)$ to a closed subspace of $L^{p}(\Omega) \times L^p(\mathbb{R}^N \times \mathbb{R}^N)$. Hence, we get the reflexivity of $X_0^{s,p}(\Omega)$ for $1\leq p<\infty$ and seperability for $1<p<\infty$, respectively.
\end{proof}

\begin{remark}
    Throughout the paper, we denote $$u_+:=\max\{u(x),0\}\text{ and } u_-:=\max\{-u(x),0\}.$$ Since $X_0^{s,p}(\Omega)$ is a vector space, for every $u\in X_0^{s,p}(\Omega)$, we have $|u|\in X_0^{s,p}(\Omega)$. Thus, $$u_\pm=\frac{1}{2}(|u|\pm u)\in X_0^{s,p}(\Omega).$$
\end{remark}
\subsection{Weak formulation} We now give the notion of solutions (weak) to our eigenvalue problem \eqref{B}.
\begin{definition}[Weak solution]
    Let $\frac{N}{sp}<q<\infty$, and assume that $V, g \in L^q(\Omega)$ where $V\geq 0, \ g > 0$ a.e. in $\Omega$. For $\lambda \in \mathbb{R}$, a function $u\in X_0^{s,p}(\Omega)$ is said to be a (weak) solution to the problem \eqref{B} if $u$ satisfies 
\begin{align}\label{B2}
    \int_{\Omega}| \nabla u |^{p-2} \nabla u \cdot \nabla v dx&+\int_{\mathbb{R}^N}\int_{\mathbb{R}^N}\frac{| u(x)-u(y)|^{p-2}(u(x)-u(y))(v(x)-v(y))}{| x-y |^{N+ps}} dxdy \nonumber\\&+\int_{\Omega}V(x)| u |^{p-2}u v dx - \lambda\int_{\Omega}g(x)| u |^{p-2}u v dx=0,
\end{align}
\\for all $v \in X_0^{s,p}(\Omega)$. We say the $u\in X_0^{s,p}(\Omega)$ is a supersolution of \eqref{B}, if the integral in \eqref{B2} is nonnegative for every nonnegative function $v\in X_0^{s,p}(\Omega)$. Similarly, a function $u$ is a subsolution of \eqref{B} if $-u$ is a supersolution of \eqref{B}. 
\end{definition}
Accordingly, we define the eigenvalues and eigenfunctions to \eqref{B}. 
\begin{definition}[Eigenvalue and principal eigenvalue]\label{evalue}
\begin{itemize}
\item[1.] A real number $\lambda$ is said to be an eigenvalue with an eigenfunction $u$ of the weighted eigenvalue problem \eqref{B} if $u \in X_0^{s,p}(\Omega)\setminus\{0\}$ is a (weak) solution of \eqref{B} correspond to $\lambda$.
\item[2.] A principal eigenvalue of the eigenvalue problem \eqref{B} is an eigenvalue of \eqref{B} with a nonnegative eigenfunction.
\end{itemize}
\end{definition}
\noindent The assumptions on the functions $V, g\in L^q(\Omega)$ for $\frac{N}{sp}<q<\infty$ is intuitive. Since the weak solutions of problem \eqref{B} are finally belongs to $W^{s,p}(\mathbb{R}^N)$ and by the Theorem \ref{FSE}, the elements of $W^{s,p}(\mathbb{R}^N)$ are integrable up to the critical power $p_s^*$, hence using the H\"older inequality with exponents $\alpha=\frac{N}{sp}$ and $\alpha'=\frac{N}{N-sp}$ we have $$\int_{\Omega}f(x)| u |^pdx<\infty\,,\,\,\text{for any } \,f\in L^{\frac{N}{ps}}(\Omega)\,\text{ and } u\in W^{s,p}(\mathbb{R}^N).$$ 
This implies that the integrals in \eqref{B2} are well-defined. Henceforth, we will use the following notation throughout the article: \begin{align}
        H_{s,p}(u,v)&:=\int_{\Omega}| \nabla u |^{p-2} \nabla u \cdot \nabla v dx\nonumber\\
        &\hspace{2cm}+\int_{\mathbb{R}^N}\int_{\mathbb{R}^N}\frac{| u(x)-u(y)|^{p-2}(u(x)-u(y))(v(x)-v(y))}{| x-y |^{N+ps}} dxdy.\label{Hsp}
    \end{align}
Note that the problem \eqref{B} is the Euler-Lagrange equation associated with the energy functional 
\begin{align}\label{B-eq-energy}
        I_\lambda(u):= \|u\|_{X_0^{s,p}(\Omega)}^p + \int_{\Omega}V(x)| u |^{p} dx  - \lambda\int_{{\Omega}}g(x)| u |^{p} dx.
    \end{align}
Hence, the problem \eqref{B} has a variational structure. We next define the {\it Rayleigh quotient} corresponding to the problem \eqref{B}, which will be used to prove the existence of weak solutions.
\begin{definition}\label{d-rayleigh}
    The {\it Rayleigh quotient} corresponding to the eigenvalue problem \eqref{B} is given by
    \begin{align}\label{B-eq-rayleigh}
        \mathcal{R} &:= \inf\limits_{u\in X_0^{s,p}(\Omega)\setminus\{0\}}\frac{ \|u\|_{X_0^{s,p}(\Omega)}^p+ \int_{\Omega}V(x)| u |^{p} dx}{\int_{\Omega}g(x)| u |^{p} dx}\nonumber \\
        &=\inf\left\{ \|u\|_{X_0^{s,p}(\Omega)}^p+ \int_{\Omega}V(x)| u |^{p} dx \mid u \in X_0^{s,p}(\Omega) \setminus\{0\}, \int_{\Omega}g(x)| u |^{p} dx =1 \right\}.
    \end{align}
\end{definition}
 
\subsection{Genus and Ljusternik-Schnirelmann category theory} We now give the notion of {\it{genus}} and discuss its fundamental properties, which will be used to prove the existence of a monotone sequence of positive eigenvalues to \eqref{B} with the help of \textit{Ljusternik-Schnirelmann category theory} \cite{S88}.
\begin{definition}[Genus]
    Let $X$ be a Banach space and $\Sigma$ be defined by 
    \begin{align*}
        \Sigma = \{A \subset X \setminus \{0\} \mid A \text{ is compact and }A=-A \}.
    \end{align*}
Then for $A \in \Sigma, $ the Genus of $A$ is given by
$$   \gamma (A)=\min \{k \in \mathbb{N} \mid \text{there exist } \phi \in C(A,\mathbb{R}^k \setminus \{0\}), \ \phi(-x)=-\phi(x) \}.$$
We say $\gamma (A)= \infty$ if the minimum does not exist.
\end{definition}
\begin{theorem}[See \cite{S88}, Proposition 2.3]\label{genus}
Let $A,B \in \Sigma$.  
\begin{enumerate}[(i)]
    \item If $x \neq 0,$ $\gamma(\{x\} \cup \{-x\})=1$.
    \item If there exist odd map $f \in C(A,B)$, then $\gamma(A) \leq \gamma(B) $.
    \item If $A \subset B$, then $\gamma(A) \leq \gamma(B)$.
    \item $\gamma(A \cup B) \leq \gamma(A) +\gamma(B)$.
    \item If $A$ is compact and $\gamma(A) < \infty$, there exist $\delta >0$ such that the neighbourhood $N_{\delta}(A) \in \Sigma$ and $\gamma(N_{\delta}(A))=\gamma(A)$.
    \item Let $Z$ be a subspace of $E$ with codimension $k$ and $\gamma(A)>k$, then $A \cap Z \neq \emptyset$.
\end{enumerate}
\end{theorem}
\noindent We conclude the present section with the following version of the {\it Ljusternik-Schnirelmann principle} \cite{L06, S88}.
\begin{theorem}\label{LS}
    Let $X$ be a real reflexive Banach space and $\lambda>0$. Let $J=T-\lambda H: X\rightarrow\mathbb{R}$, where $T,H$ are $C^1$, even functionals with $T(0)=H(0)=0$ and $T$ is bounded from below on $X$.  Suppose, $T$ satisfies the {\it Palais-Smale condition} on $\mathcal{M}=\{u \in X: H(u)=1\}$. Then the problem 
    \begin{equation*}
        T'(u)=\lambda H'(u)
    \end{equation*}
    has a sequence of eigenvalues $(\lambda_n)$ defined as a critical values of $T$, given by
    \begin{equation*}
        \lambda_n=\inf_{A \in \Sigma_n} \max_{u \in A} T(u)
    \end{equation*}
    where $\Sigma_n= \{ A \in \Sigma \mid \gamma(A) \geq n \}$ and $\Sigma = \{A \subset X \setminus \{0\} \mid A \text{ is compact and }A=-A \}$. 
\end{theorem}
\section{Comparison principles and useful estimates}
Here, we develop comparison principles and a logarithmic estimate related to our eigenvalue problem, which will be used several times throughout the paper. We first prove the following {\it weak comparison principle} for weak solutions of \eqref{B}.
\begin{theorem}[\textbf{Weak comparison Principle}]\label{B-CP}
    Let $u, v \in X_0^{s,p}(\Omega)$ such that 
    \begin{equation}\label{B-CP1}
        H_{s,p}(u,\psi)+\int_{\Omega}V|u|^{p-2}u\psi \leq H_{s,p}(v,\psi)+\int_{\Omega}V|v|^{p-2}v\psi 
    \end{equation}
    for all $\psi \in X_0^{s,p}(\Omega)$ such that $\psi \geq 0$ and $u \leq v$ in $\mathbb{R}^N \setminus \Omega$. Then, $u \leq v$ in $\mathbb{R}^N$.
\end{theorem}
\begin{proof}
    Taking $\psi= (u-v)_+$ in \eqref{B-CP1}, we get
    \begin{equation}\label{B-CP2}
        I_1+I_2+I_3 \geq 0,
    \end{equation}
    where
    \begin{align}\label{B-CP-I}
        &I_1=\int_{\Omega}\left(| \nabla v |^{p-2} \nabla v  -| \nabla u |^{p-2} \nabla u \right)\cdot\nabla (u-v)_+dx ,\nonumber \\
        &I_2 =\int_{\mathbb{R}^N}\int_{\mathbb{R}^N} k(u,v)((u-v)_+(x)-(u-v)_+(y)) dxdy, \nonumber \\
       \,\,\text{and }\,  &I_3= \int_{\Omega}(|v|^{p-2}v-|u|^{p-2}u)(u-v)_+ dx.
    \end{align}
    Here $\displaystyle{k(u,v):=\frac{(| v(x)-v(y)|^{p-2}(v(x)-v(y)-| u(x)-u(y)|^{p-2}(u(x)-u(y))} {| x-y |^{N+ps}}}$.
    Next, using the inequality
    \begin{equation}\label{B-eq-ab}
        | b| ^{p-2}b-| a| ^{p-2}a=(p-1)(b-a)\int_0^1 | a+s(b-a)| ^{p-2}ds,
    \end{equation}
     for $a=\nabla v, b=\nabla u$ in $I_1$, $a= v(x)-v(y), b=u(x)-u(y)$ in $I_2$, and $a=v, b=u$ in $I_3$ respectively and then combining them in the equation \eqref{B-CP2}, we obtain
    \begin{equation}\label{B-CP-3}
        J_1+J_2+J_3 \geq 0,
    \end{equation}
    where 
    \begin{align}\label{B-CP-J1}
        J_1&=(p-1)\int_{\Omega}\nabla(u-v)_+ \cdot \nabla(v-u)\int_0^1 | \nabla u+t\nabla(v-u)| ^{p-2}dtdx \nonumber \\
        &=-(p-1)\int_{\Omega}|\nabla(u-v)_+|^2\int_0^1 | \nabla u+t\nabla(v-u)| ^{p-2}dtdx \leq 0,
    \end{align}
    \begin{align}\label{B-CP-J2}
        J_2 &=(p-1)\int_{\mathbb{R}^N}\int_{\mathbb{R}^N}\frac{(v(y)-v(x)-(u(y)-u(x)))}{| x-y |^{N+ps}}((u(y)-v(y))_+ -(u(x)-v(x))_+) \nonumber\\
      &\hspace{3cm}\times \int_0^1| u(y)-u(x)+t(v(y)-v(x)-u(y)+u(x))| ^{p-2} dt dx dy \nonumber \\
      & \leq -(p-1)\int_{\mathbb{R}^N}\int_{\mathbb{R}^N}\frac{(((u(y)-v(y))_+ -(u(x)-v(x))_+))^2}{| x-y |^{N+ps}}\nonumber\\
      &\hspace{2cm}\times \int_0^1| u(y)-u(x)+t(v(y)-v(x)-u(y)+u(x))| ^{p-2} dt dx dy \leq 0,
    \end{align}
    and 
    \begin{align}\label{B-CP-J3}
        J_3&=(p-1)\int_{\Omega}(u-v)_+(v-u)\int_0^1|  u+t(v-u)| ^{p-2}dtdx \nonumber \\
        &=-(p-1)\int_{\Omega}|(u-v)_+|^2\int_0^1|  u+t(v-u)| ^{p-2}dtdx\leq 0.
    \end{align}
   Since all the integral values $J_1, J_2$ and $J_3$ are non-positive, hence from \eqref{B-CP-3} we have $J_1+J_2+J_3=0$. This implies that $\nabla(u-v)_+=0$ (from $J_1$) and $((u(y)-v(y))_+ =(u(x)-v(x))_+$ a.e. (from $J_2$) in the set $\{x:u(x)\geq v(x) \}$ and from $J_3$, $(u-v)_+=0$ a.e. in $\Omega$. Thus it follows that $u(x)\leq v(x)$ a.e. in $\Omega$. Since $u=v=0$ on $\mathbb{R}^N \setminus \Omega$, therefore we get $v(x) \geq u(x)$ a.e in $\mathbb{R}^N$.  
   \end{proof}
\noindent The following logarithmic estimate is crucial in obtaining a strong comparison principle.
\begin{lemma}\label{B-lem1}
   Assume $x_0\in \Omega$, and let $R>0$ such that $B_{R}(x_0) \subset \Omega$. Let $u\in X_0^{s,p}(\Omega)$ be a weak supersolution to the problem 
    \begin{align}\label{B1}
    \begin{split}
        L_{s,p}u +V|u|^{p-2}u &=0 \ \text{in } \Omega  \\
        u&=0 \ \text{in } \mathbb{R}^N \setminus \Omega.
        \end{split}
    \end{align}
     Suppose $u \geq 0$ in $B_R(x_0)$. Then, for $0<r<\frac{R}{2}$, there exist a constant $C=C(p)>0$ such that for $\delta>0$,
    \begin{align}\label{B-l1}
        \int_{B_r(x_0)} |\nabla &\log(u+\delta)|^p dx 
+  \int_{B_r(x_0)}\int_{B_r(x_0)}  \frac{1}{|x-y|^{N+sp}}\left|\log\left( \frac{u(x)+\delta)}{u(y)+\delta} \right)\right|^p dy dx \nonumber \\
 &\leq Cr^N(r^{-p}+r^{-sp}+\delta^{1-p} R^{-sp}\, \operatorname{Tail}(u_-;x_0, R)^{p-1}) +r^{\frac{Np-N+sp}{p}}\|V\|_{L^{p_s^*}(\Omega)}.
    \end{align}
\end{lemma}
\begin{proof}
     Since $u$ is a weak supersolution to the problem \eqref{B1}, hence for any $v \in X_0^{s,p}(\Omega)$ such that $v \geq 0$, we have
    \begin{equation}\label{B-w supsln}
        H_{s,p}(u,v)+\int_{\Omega}V|u|^{p-2}u v dx \geq 0.
    \end{equation}
     Consider $0<r<\frac{R}{2}$ and choose $\psi \in C_c^{\infty}(B_{3r/2}(x_0))$ such that $0\leq \psi \leq 1$ and $\psi =1$ in $B_{r}(x_0)$ with $|\nabla \psi| \leq \frac{c}{r}$ for some constant $c>0$. Let $\delta>0$ be any real number. Then taking the test function $v=(u+\delta)^{1-p}\psi^p$ in the equation \eqref{B-w supsln}, we arrived at 
    \begin{align}\label{B-l2}
        I_1+I_2+I_3+I_4 \geq 0,
    \end{align}
    where 
    \begin{align}\label{B-i123}
        I_1 &= \int_{B_{2r}(x_0)}|\nabla u|^{p-2} \nabla u\cdot \nabla v dx,\nonumber \\
        I_2 &= \int_{B_{2r}(x_0)}\int_{B_{2r}(x_0)} \frac{|u(x)-u(y)|^{p-2}}{|x-y|^{N+sp}} (u(x)-u(y))(v(x)-v(y)) dx dy, \,\nonumber \\
        I_3 &= 2 \int_{\mathbb{R}^N \setminus B_{2r}(x_0)}\int_{B_{2r}(x_0)} \frac{|u(x)-u(y)|^{p-2}}{|x-y|^{N+sp}}(u(x)-u(y))v(x) dx dy, \,\nonumber \\
      \text{and }  I_4 &= \int_{B_{2r}(x_0)} V|u|^{p-2}u v dx.
    \end{align}

    \noindent We now derive the estimates separately for each integral $I_i$ for $i=1,2,3,4$. To obtain an estimate for the integral $I_1$, we follow similar steps as in the proof of Lemma 3.4 in \cite{KK07} and use the properties of $\psi$. For completeness, we provide detailed steps. Substituting 
    $$\nabla v=p(u+\delta)^{1-p}\psi^{p-1}\nabla \psi -(p-1)\psi^p (u+\delta)^{-p}\nabla u \text{ in } I_1$$ we get
    \begin{align}\label{B-LI1'}
        I_1 &= -(p-1)\int_{B_{2r}(x_0)}\frac{|\nabla u|^p}{(u+\delta)^p}\psi^p dx+p\int_{B_{2r}(x_0)}|\nabla u|^{p-2} \frac{\psi^{p-1}}{(u+\delta)^{p-1}}\nabla u \cdot \nabla \psi dx \nonumber\\
        &= -(p-1)\int_{B_{2r}(x_0)}\frac{|\nabla (u+\delta)|^p}{(u+\delta)^p}\psi^p dx \nonumber\\
        &\hspace{2.5cm}+p\int_{B_{2r}(x_0)}|\nabla (u+\delta)|^{p-2} \frac{\psi^{p-1}}{(u+\delta)^{p-1}}\nabla (u+\delta) \cdot \nabla \psi dx \nonumber\\
        &= -(p-1)\int_{B_{2r}(x_0)}|\nabla \log(u+\delta)|^p\psi^p dx\nonumber\\
        &\hspace{2.5cm}+p\int_{B_{2r}(x_0)}|\nabla \log(u+\delta)|^{p-2} \psi^{p-1}\nabla \log(u+\delta) \cdot \nabla \psi dx.
    \end{align}
Using the Young's inequality in the second integral of the RHS of \eqref{B-LI1'} with the exponents $\frac{p}{p-1}$ and $p$, we have
    \begin{align}\label{B-LI1''}
        \int_{B_{2r}(x_0)}&|\nabla \log(u+\delta)|^{p-2} \psi^{p-1}\nabla \log(u+\delta) \cdot \nabla \psi dx \nonumber\\
        &\leq \epsilon \int_{B_{2r}(x_0)}|\nabla \log(u+\delta)|^p\psi^p dx + C(\epsilon) \int_{B_{2r}(x_0)} |\nabla \psi|^p dx,
    \end{align}
    for any $\epsilon>0$. Now, choose $\epsilon>0$ such that $p\epsilon<p-1$. Note that $|\nabla \psi|\leq \frac{c}{r}$, $\psi\equiv1$ in $B_r(x_0)$ and $\psi \geq 0$. Thus substituting \eqref{B-LI1''} in \eqref{B-LI1'}, we get
    \begin{align}\label{B-LI1}
        I_1 &\leq -C_1 \int_{B_{2r}(x_0)}|\nabla \log(u+\delta)|^p\psi^p dx +C\int_{B_{2r}(x_0)} |\nabla \psi|^p dx \nonumber\\
        &\leq -C_1 \int_{B_{r}(x_0)}|\nabla \log(u+\delta)|^p\psi^p dx + C \int_{B_{2r}(x_0)}\frac{c}{r^p} dx \nonumber\\
        &\leq  -C_1 \int_{B_{r}(x_0)}|\nabla \log(u+\delta)|^p\psi^p dx + \frac{C}{r^p}|B_{2r}(x_0)| \nonumber\\
        &=-C_1 \int_{B_{r}(x_0)}|\nabla \log(u+\delta)|^p\psi^p dx + Cr^{N-p}.
    \end{align}

Next, using the estimate for integrals $I_2$  obtained in proof of the Lemma 1.3 in \cite{DKP16} and applying the fact that $\psi \equiv 1$ in $B_r(x_0)$ and $\psi \geq0$, we get 
    \begin{align}\label{B-LI2'}
        I_2 &\leq C \iint_{B_{2r}(x_0)\times B_{2r}(x_0)} \left[-\frac{\left|\log\left( \frac{u(x)+\delta)}{u(y)+\delta} \right)\right|^p}{|x-y|^{N+sp}} \psi(y)^p  + \frac{|\psi(x)-\psi(y)|^p}{|x-y|^{N+sp}}  \right]dy dx \nonumber \\
        & \leq -C \!\int_{B_{r}(x_0)}\int_{B_{r}(x_0)} \frac{\left|\log\left( \frac{u(x)+\delta)}{u(y)+\delta} \right)\right|^p}{|x-y|^{N+sp}} dy dx +C \!\int_{B_{2r}(x_0)}\int_{B_{2r}(x_0)} \frac{|\psi(x)-\psi(y)|^p}{|x-y|^{N+sp}}  dy dx. 
    \end{align}
    Now, since $|\nabla \psi| \leq \frac{c}{r}$, by the mean value theorm, $|\psi(x)-\psi(y)| \leq \frac{c}{r}|x-y|$ for all $x,y \in \mathbb{R}^N$. Thus, we have
    \begin{align}\label{B-L12''}
        \int_{B_{2r}(x_0)}\int_{B_{2r}(x_0)} \frac{|\psi(x)-\psi(y)|^p}{|x-y|^{N+sp}}  dy dx &\leq \frac{C}{r^p}\int_{B_{2r}(x_0)}\int_{B_{2r}(x_0)} \frac{1}{|x-y|^{N-p(1-s)}} \nonumber\\
        &\leq \frac{C}{r^p}|B_{2r}(x_0)|r^{p(1-s)} \leq C r^{N-sp}.
    \end{align}
    Hence, by the equations \eqref{B-LI2'} and \eqref{B-L12''}, we get
    \begin{equation}\label{B-LI2}
        I_2 \leq  -C \int_{B_{r}(x_0)}\int_{B_{r}(x_0)} \frac{1}{|x-y|^{N+sp}}\left|\log\left( \frac{u(x)+\delta)}{u(y)+\delta} \right)\right|^p dy dx +C r^{N-sp}.
    \end{equation}
    Similarly, using the estimate obtained from the proof of Lemma 1.3 in \cite{DKP16} and applying \eqref{B-L12''}, we get
    \begin{align}\label{B-LI3}
        I_3 &\leq C \int_{B_{2r}(x_0)}\int_{B_{2r}(x_0)} \frac{|\psi(x)-\psi(y)|^p}{|x-y|^{N+sp}}  dy dx + Cr^{N-sp}+ C\delta^{1-p} r^N R^{-sp}\,\operatorname{Tail}(u_-;x_0, R)^{p-1} \nonumber\\
        &\leq C\delta^{1-p} r^N R^{-sp}\,\operatorname{Tail}(u_-;x_0, R)^{p-1}+ Cr^{N-sp}.
    \end{align}
    Note that the constant $C>0$ depending on $p$, may vary line by line. Finally, using the fact that $0 \leq \frac{u}{u+\delta}\leq 1$ in $B_R(x_0)$ and $B_{2r}(x_0)\subset B_R(x_0)$, we apply the H\"older's inequality with the exponents $p_s^*$ and $\frac{p_s^*}{p_s^*-1}=\frac{Np}{Np-N+sp}$, to get 
    \begin{align}\label{B-LI4}
        I_4 = \int_{B_{2r}(x_0)} V u^{p-1} (u+\delta)^{1-p}\psi dx 
         &\leq \|V\|_{L^{p_s^*}(\Omega)} |B_{2r}(x_0)|^{\frac{Np-N+sp}{Np}} \nonumber\\
         &=C \|V\|_{L^{p_s^*}(\Omega)} r^{\frac{Np-N+sp}{p}}.
    \end{align}
    By combining \eqref{B-LI1}, \eqref{B-LI2}, \eqref{B-LI3}, and \eqref{B-LI4} in \eqref{B-l2}, we obtain a constant $C>0$ depending on $p$ such that
    \begin{align*}
          \int_{B_r(x_0)} |\nabla &\log(u+\delta)|^p dx 
+  \int_{B_r(x_0)}\int_{B_r(x_0)}  \frac{1}{|x-y|^{N+sp}}\left|\log\left( \frac{u(x)+\delta)}{u(y)+\delta} \right)\right|^p dy dx \nonumber \\
 &\leq Cr^N(r^{-p}+r^{-sp}+\delta^{1-p} R^{-sp} \,\operatorname{Tail}(u_-;x_0, R)^{p-1}) +r^{\frac{Np-N+sp}{p}}\|V\|_{L^{p_s^*}(\Omega)}.
    \end{align*}
Hence the lemma.
\end{proof}
\noindent As a consequence of Lemma \ref{B-lem1}, we have the generalized strong maximum principle.
\begin{theorem}[\textbf{Strong Maximum Principle}]\label{B-smp}
    Let $u \geq 0$ in $\Omega$ and $u=0$ in $\mathbb{R}^N\setminus\Omega$ be an eigenfunction of \eqref{B} associated with an eigenvalue $\lambda>0$. Then $u>0$ in $\Omega$.
\end{theorem}
\begin{proof}
    Since $g>0$ a.e. in $\Omega$ and $u\geq 0$ is an eigenfunction corresponding to $\lambda>0$, hence we have 
    \begin{align*}
        H_{s,p}(u,v)+\int_{\Omega}V|u|^{p-2}u v dx &= \lambda\int_{\Omega} g|u|^{p-2}uv dx \geq 0,
    \end{align*}
    for any $v \in X_0^{s,p}(\Omega)$ with $v \geq 0$.
    This shows that the eigenfunction $u\geq 0$ is a weak supersolution to the problem \eqref{B1}.\\
    Suppose the set $A:=\{x \in \Omega\,: u(x)=0\}$ has positive measure. Now choose $x_0 \in \Omega$ and  a real number $R>0$ such that
    $B_R(x_0) \subset \Omega$, and there exists $r\in(0,\frac{R}{2})$ such that $|B_r(x_0) \cap A|>0$. Since $u \geq 0$, we have $$\operatorname{Tail}(u_-;x_0, R)=0.$$
    Thus by Lemma \ref{B-lem1}, for any $\delta>0$, we get
    \begin{align}\label{B-smp1}
         \int_{B_r(x_0)} |\nabla \log(u+\delta)|^p dx 
         +  \iint_{B_r(x_0)\times B_r(x_0)}  \frac{1}{|x-y|^{N+sp}}\left|\log\left( \frac{u(x)+\delta)}{u(y)+\delta} \right)\right|^p dy dx \leq M,
  \end{align}
  where $$M= Cr^N(r^{-p}+r^{-sp}) +r^{\frac{Np-N+sp}{N}}\|V\|_{L^{p_s^*}(\Omega)}.$$ 
  For $\delta>0$, we define 
  $$G_{\delta}(x)=\log(1+\frac{u(x)}{\delta}), \ x\in \Omega.$$
  For any $x\in B_r(x_0)$, $y \in A\cap B_r(x_0)$ we have $|x-y|\leq 2r < R$. Then $G_\delta(y) = 0$. Thus,
  \begin{align}\label{B-smp2}
      |G_{\delta}(x)|^p = |G_{\delta}(x)-G_{\delta}(y)|^p &= \left|\log\left( \frac{u(x)+\delta}{u(y)+\delta}\right)\right|^p \nonumber \\
      & \leq \frac{R^{N+sp}}{|x-y|^{N+sp}}\left|\log\left(\frac{u(x)+\delta}{u(y)+\delta}\right)\right|^p.
  \end{align}
  Therefore, on taking average for all $y \in A \cap B_r(x_0)$, we obtain
  \begin{align}\label{B-smp3}
      |G_{\delta}(x)|^p  &\leq \frac{R^{N+sp}}{|A \cap B_r(x_0)|}\int_{A \cap B_r(x_0)}  \frac{1}{|x-y|^{N+sp}}\left|\log\left( \frac{u(x)+\delta)}{u(y)+\delta} \right)\right|^p dy \nonumber \\
      &\leq \frac{R^{N+sp}}{|A \cap B_r(x_0)|}\int_{B_r(x_0)}  \frac{1}{|x-y|^{N+sp}}\left|\log\left( \frac{u(x)+\delta)}{u(y)+\delta} \right)\right|^p dy.
  \end{align}
  Thus, from \eqref{B-smp1} and \eqref{B-smp3}, we get
  \begin{align}\label{B-smp4}
      \int_{B_r(x_0)}|G_{\delta}(x)|^p dx &\leq \frac{R^{N+sp}}{|A \cap B_r(x_0)|}\iint\limits_{B_r(x_0)\times B_r(x_0)}  \frac{1}{|x-y|^{N+sp}}\left|\log\left( \frac{u(x)+\delta)}{u(y)+\delta} \right)\right|^p dy dx \nonumber \\
      &\leq C_r \left(\int_{B_r(x_0)} |\nabla \log(u+\delta)|^p dx +  \iint\limits_{B_r(x_0)\times B_r(x_0)}  \frac{|\log\left( \frac{u(x)+\delta)}{u(y)+\delta} \right)|^p}{|x-y|^{N+sp}} dy dx \right) \nonumber \\
         &\leq C_rM,
  \end{align}
  where $C_r=\frac{R^{N+sp}}{|A \cap B_r(x_0)|}$. Now for $\delta \rightarrow 0$, $G_\delta(x) \rightarrow \infty$ for all $x \in B_r(x_0)$ with $u(x)>0$. Thus, on taking $\delta \rightarrow 0$ in the equation \eqref{B-smp4}, we get that $u=0$ a.e. in $B_r(x_0)$, which is a contradiction. Therefore $|A|=0$. Hence, we conclude $u>0$ in $\Omega$.
\end{proof}
\section{Existence and qualitative properties of principal eigenvalue}
In this section, we will show the existence of principal eigenvalue and study some of its qualitative properties. We begin with the proof of the existence of a positive eigenvalue to the problem \eqref{B}, which is the least among all other positive eigenvalues. We call it the {\it first eigenvalue} and denote it by $\lambda_1$. Later, we show the {\it first eigenvalue is precisely the principal eigenvalue} to \eqref{B}.
\begin{theorem}\label{B-thm1}
    The Rayleigh quotient $\mathcal{R}$ defined in \eqref{B-eq-rayleigh} is non-zero and attained at a $\phi \in X_0^{s,p}(\Omega) \setminus \{0\}$. Moreover, the first eigenvalue $\lambda_1$ is the Rayleigh quotient $\mathcal{R}$ and $\phi$ is an associated eigenfunction.
\end{theorem}
\begin{proof}
    Consider a sequence $(u_n)$ in $X_0^{s,p}(\Omega) \setminus \{0\}$ such that $\displaystyle{\int_{\Omega} g(x)| u_n |^{p} dx =1}$ and
    \begin{equation}\label{B1'}
         \int_{\Omega}| \nabla u_n |^{p}dx+ \int_{\mathbb{R}^N}\int_{\mathbb{R}^N}\frac{| u_n(x)-u_n(y)|^{p}}{| x-y |^{N+ps}} dxdy + \int_{\Omega}V(x)| u_n |^{p} dx \rightarrow \mathcal{R} 
    \end{equation}
as $n \rightarrow \infty$. 
    Since 
    $$\|u_n\|_{X_0^{s,p}(\Omega)}^p \leq H_{s,p}(u_n,u_n) + \int_{\Omega}V(x)| u_n |^{p} dx $$
    for all $n \in \mathbb{N}$, the sequence $(u_n)$ is a bounded in $X_0^{s,p}(\Omega)$. Since $X_0^{s,p}(\Omega)$ is a reflexive Banach space, hence up to a subsequence, $u_n\rightharpoonup \phi$ weakly in $X_0^{s,p}(\Omega)$. Further using the compact embedding \eqref{CPT MET}, up to a subsequence, $u_n \rightarrow \phi$ in $L^q(\Omega)$ for all $1 \leq q  <p^*$ and $u_n \rightarrow \phi$ a.e. in $\Omega$. Since $p_s^* <p^*$, by \cite[Theorem 4.9]{B11}, it follows that there exists $h \in L^{p_s^*}(\Omega)$ such that $|u_n| \leq h$ a.e. in $\Omega$.
   As $g \in L^{\frac{N}{ps}}(\Omega)$, using the H\"older's inequality we have
    \begin{equation}\label{B3}
        \left| \int_{\Omega}g(x)| u_n |^{p} dx - \int_{\Omega}g(x)| \phi |^{p} dx\right| \leq \|g\|_{L^{\frac{N}{ps}}(\Omega)} \left(\int_{\Omega}| | u_n |^{p}-| \phi |^{p} |^{\frac{N}{N-ps}} dx \right)^{\frac{N-ps}{N}}.
    \end{equation}
    Further, $|| u_n |^{p}-| \phi |^{p} |^{\frac{N}{N-ps}} \leq 2^{\frac{N}{N-ps}}h^{p_s^*}$. Thus, applying the dominated convergence theorem to the equation \eqref{B3}, we get
    $$\left| \int_{\Omega}g(x)| u_n |^{p} dx - \int_{\Omega}g(x)| \phi |^{p} dx\right| \rightarrow 0 \text{ as } n \rightarrow \infty.$$
    Hence $\displaystyle{\int_{\Omega}g(x)| \phi |^{p} dx =1}$. This implies that $\phi \neq 0$ is an element of $X_0^{s,p}(\Omega)$. Similarly, proceeding as before, one can obtain 
    \begin{equation}\label{B4}
        \int_{\Omega}V(x)| u_n |^{p} dx \rightarrow  \int_{\Omega}V(x)| \phi |^{p} dx \text{ as } n \rightarrow \infty.
    \end{equation}
    On the other hand, by the property of weak convergence, we have
    \begin{equation}\label{B5}
        \|\phi\|_{X_0^{s,p}(\Omega)}=H_{s,p}(\phi,\phi) \leq \liminf\limits_{n \rightarrow \infty} \|u_n\|_{X_0^{s,p}(\Omega)}=\liminf\limits_{n \rightarrow \infty} H_{s,p}(u_n,u_n).
    \end{equation}
    Thus combining \eqref{B1'}, \eqref{B4} and \eqref{B5}, we get
    \begin{equation}\label{B6}
        \int_{\Omega}| \nabla \phi|^{p}dx+ \int_{\mathbb{R}^N}\int_{\mathbb{R}^N}\frac{| \phi(x)-\phi(y)|^{p}}{| x-y |^{N+ps}} dxdy + \int_{\Omega}V(x)| \phi |^{p} dx \leq \mathcal{R}.
    \end{equation}
    Since $\displaystyle{\int_{\Omega}g(x)| \phi |^{p} dx=1}$, hence the definition of Rayleigh quotient $\mathcal{R}$ implies that 
    \begin{equation}\label{B7}
        \int_{\Omega}| \nabla \phi|^{p}dx+ \int_{\mathbb{R}^N}\int_{\mathbb{R}^N}\frac{| \phi(x)-\phi(y)|^{p}}{| x-y |^{N+ps}} dxdy + \int_{\Omega}V(x)| \phi |^{p} dx = \mathcal{R}.
    \end{equation}
    This shows that the {\it Rayleigh quotient} $\mathcal{R}$ is attained at $\phi \in X_0^{s,p}(\Omega)\setminus\{0\}$ and $\mathcal{R} \neq 0$. Now consider any eigenvalue $\lambda$ to the problem \eqref{B} and an associated eigenfunction $v$. Taking $v$ as the test function in the weak formulation \eqref{B2}, we get
    \begin{align*}
        \|v\|_{X_0^{s,p}(\Omega)}^p +\int_\Omega V|v|^p dx =\lambda \int_{\Omega}g(x)| v|^{p} dx \nonumber \\
        \implies \lambda= \frac{\|v\|_{X_0^{s,p}(\Omega)}^p+ \int_\Omega V|v|^p dx}{\int_{\Omega}g(x)| v|^{p} dx} \geq \mathcal{R}.
    \end{align*}
    Thus, by the definition of $\mathcal{R}$, we obtain that the first eigenvalue $\lambda_1=\mathcal{R}$ and $\phi$ is an associated eigenfunction.
\end{proof}

The next theorem asserts that the eigenfunctions corresponding to the eigenvalue given by the {\it Rayleigh quotien}t do not change their sign, i.e., they cannot have zeros in the domain. Thus, it follows that the eigenvalue given by the {\it Rayleigh quotient} is a {\it principal eigenvalue}.
\begin{theorem}\label{B-thm2}
    Let $u \in X_0^{s,p}(\Omega)$ be an eigenfunction associated with the eigenvalue $\lambda_1$. Then either $u>0$ or $u<0$ in $\Omega$.
\end{theorem}
\begin{proof}
    Without loss of generality, assume $\int_{\Omega}g(x)| u |^{p} dx=1$. By, the Theorem \ref{B-thm1}, we have $$\int_{\Omega}| \nabla u|^{p}dx+ \int_{\mathbb{R}^N}\int_{\mathbb{R}^N}\frac{| u(x)-u(y)|^{p}}{| x-y |^{N+ps}} dxdy + \int_{\Omega}V(x)| u |^{p} dx = \mathcal{R}.$$ 
    Let us assume that $u$ changes sign in $\Omega$. Then there exists subsets $\omega_1,\omega_2$ of $\Omega$, both of positive measure such that $u(x)u(y)<0$ for $x \in \omega_1,\ y \in \omega_2$. Note that for any $a,b \in \mathbb{R}$, we have
    $$||a|-|b|| \leq |a-b|, $$  and strict inequality holds when $ab<0$. So using the above inequality, we obtain
    \begin{equation*}
        \int_{\omega_1}\int_{\omega_2}\frac{| |u(x)|-|u(y)||^{p}}{| x-y |^{N+ps}} dxdy < \int_{\omega_1}\int_{\omega_2}\frac{| u(x)-u(y)|^{p}}{| x-y |^{N+ps}} dxdy.
    \end{equation*}
    By the property of Sobolev spaces (see (\cite[Lemma 7.6]{GT98}), we also have $|u| \in X_0^{s,p}(\Omega)$ and $\nabla | u|=| \nabla u| $ a.e. in $\Omega$. Thus
    \begin{align*}
         \int_{\Omega}| \nabla |u||^{p}  dx&+ \int_{\mathbb{R}^N}\int_{\mathbb{R}^N}\frac{| | u(x)| -| u(y)| |^{p}}{| x-y |^{N+ps}} dxdy +\int_{\Omega} V(x)|u|^p dx \\
        &<\int_{\Omega}| \nabla u |^{p}  dx+ \int_{\mathbb{R}^N}\int_{\mathbb{R}^N}\frac{| u(x)-u(y)|^{p}}{| x-y |^{N+ps}} dxdy +\int_{\Omega} V(x)|u|^p dx  =\lambda_1 =\mathcal{R},
    \end{align*}
    which contradicts to the definition of Rayleigh quotient $\mathcal{R}$. Therefore, $u$ does not change the sign in $\Omega$. This implies that either $u\geq 0$ or $u\leq 0$ in $\Omega$. Now if $u \geq 0$ in $\Omega$, by the Theorem \ref{B-smp}, we have $u>0$ in $\Omega$. Similarly, the case $u<0$ can be seen by replacing the above arguments for $u$ with $(-u)$.
\end{proof}
The following theorem says that $\lambda_1$ is the only eigenvalue possessing a nonnegative eigenfunction associated with it, i.e., $\lambda_1$ is the only {\it principal eigenvalue}. In other words, non-principal eigenfunctions change signs in their domain.
\begin{theorem}\label{B-thm3}
    Let $v$ be an eigenfunction associated with an eigenvalue $\lambda>\lambda_1$. Then $v$ changes sign in $\Omega$.
\end{theorem}
\begin{proof}
    Suppose $v$ does not change sign in $\Omega$. Then without loss of generality, let us assume $v>0$ in $\Omega$ such that $\displaystyle{\int_{\Omega} g(x)|v|^p dx=1}$. Consider an eigenfunction $u>0$ associated with principal eigenvalue $\lambda_1$ such that $\displaystyle{\int_{\Omega} g(x)|u|^p dx=1}$. For $t \in (0,1)$, define $q_t=(tu^p+(1-t)v^p)^{\frac{1}{p}}.$  Obviously, 
    \begin{align}
        \int_{\Omega}g|q_t|^p dx &= t\int_{\Omega}g|u|^p dx+(1-t)\int_{\Omega}g|v|^p dx =1, \label{B8}\\
        \text{and }\int_{\Omega}V|q_t|^p dx &= t\int_{\Omega}V|u|^p dx+(1-t)\int_{\Omega}V|v|^p dx. \label{B8'}
    \end{align}
    Applying the convexity of the map $t \mapsto t^p$ for $p>1$, it follows that
    \begin{align}\label{B10}
        | \nabla q_t | ^p &= \left|(tu^p+(1-t)v^p)^{\frac{1}{p}-1}(tu^{p-1}\nabla u+(1-t)v^{p-1}\nabla v)\right|^p \nonumber \\
        &= q_t^p \left| t \frac{u^p \nabla u}{q_t^p u} +(1-t)\frac{v^p \nabla v}{q_t^p v} \right|^p \nonumber \\
        &= q_t^p |\left| w \frac{\nabla u}{u} +(1-w) \frac{\nabla v}{v}\right|^p,\ \text{where } w=\frac{t u^p}{t u^p+(1-t)v^p} \nonumber \\
        & \leq q_t^p \left(  w \left|\frac{\nabla u}{u}\right|^p +(1-w) \left| \frac{\nabla v}{v}\right|^p\right) 
        =t|\nabla u|^p +(1-t) | \nabla v|^p.
    \end{align}
    Thus we get
    \begin{equation}\label{B11}
        \int_{\Omega}| \nabla q_t | ^p dx \leq t \int_{\Omega}| \nabla u | ^p dx + (1-t) \int_{\Omega}| \nabla v | ^p dx.
    \end{equation}
    In addition, by using \cite[Lemma 4.1]{FP14} we have 
    \begin{align}\label{B12}
         &\int_{\mathbb{R}^N}\int_{\mathbb{R}^N}\frac{| q_t(x)-q_t(y)|^{p}}{| x-y |^{N+ps}} dxdy  \nonumber \\
         &\leq  t  \int_{\mathbb{R}^N}\int_{\mathbb{R}^N}\frac{| u(x)-u(y)|^{p}}{| x-y |^{N+ps}} dxdy +(1-t)  \int_{\mathbb{R}^N}\int_{\mathbb{R}^N}\frac{| v(x)-v(y)|^{p}}{| x-y |^{N+ps}} dxdy.
    \end{align}
    Hence using the inequalities \eqref{B8'}, \eqref{B11} and \eqref{B12}, we get
    \begin{align}\label{B13}
        H_{s,p}(q_t,q_t) +\int_{\Omega}V|q_t|^p dx \leq tH_{s,p}(u,u) & +t\int_{\Omega}V|u|^p dx +(1-t)H_{s,p}(v,v) \nonumber \\
        &+\int_{\Omega}(1-t)V|v|^p dx.
    \end{align}
Again, from this inequality, we arrive at
    \begin{align}\label{B14}
        & H_{s,p}(q_t,q_t) +\int_{\Omega}V|q_t|^p dx - \left( H_{s,p}(v,v)+\int_{\Omega}V|v|^p dx \right) \nonumber \\
        &\leq t\left( H_{s,p}(u,u)+\int_{\Omega}V|u|^p dx \right)-t\left(H_{s,p}(v,v)+\int_{\Omega}V|v|^p dx\right) \nonumber \\
        &= t(\lambda_1-\lambda).
    \end{align}
By the convexity of the map $t \mapsto t^p$, we have
    \begin{align}
       \int_{\Omega}V|q_t|^p dx-\int_{\Omega}V|v|^p dx &\geq p \int_{\Omega}V|v|^{p-2}v(q_t-v) dx,\label{B9} \\
        \int_\Omega |\nabla q_t|^p dx-\int_\Omega |\nabla v|^p dx  &\geq p \int_\Omega |\nabla v|^{p-2}\nabla v \cdot \nabla (q_t-v) dx,\label{B-14'}
        \end{align}
        and 
        \begin{align}
        &\int_{\mathbb{R}^N}\int_{\mathbb{R}^N}\frac{| q_t(x)-q_t(y)|^{p}}{| x-y |^{N+ps}} dxdy - \int_{\mathbb{R}^N}\int_{\mathbb{R}^N}\frac{| v(x)-v(y)|^{p}}{| x-y |^{N+ps}} dxdy \nonumber\\ 
        & ~~~~\geq p \int_{\mathbb{R}^N}\int_{\mathbb{R}^N}\frac{| v(x)-v(y)|^{p-2}}{| x-y |^{N+ps}}(v(x)-v(y))((q_t-v)(x)-(q_t-v)(y)) dxdy.\label{B-14''}
    \end{align}
    Thus, from the equations \eqref{B-14'} and \eqref{B-14''}, we have
    \begin{equation}\label{B-14-1}
        H_{s,p}(q_t,q_t)-H_{s,p}(v,v) \geq p H_{s,p}(v,q_t-v).
    \end{equation}
    Combining the inequalities \eqref{B9} and \eqref{B-14-1} in \eqref{B14} and using the fact that $v$ is an eigenfunction associated to the eigenvalue $\lambda$, we obtain
    \begin{align}\label{B-14-3}
        0> t(\lambda_1-\lambda) & \geq p \left( H_{s,p}(v,q_t-v)+\int_{\Omega}V|v|^{p-2}v(q_t-v) dx \right) \nonumber \\
        &= p \lambda \int_{\Omega}g|v|^{p-2}v(q_t-v) dx.
    \end{align}
    This implies that 
    \begin{equation}\label{B-14-4}
        \frac{p \lambda}{t}\int_{\Omega}g|v|^{p-2}v(q_t-v) dx \leq \lambda_1-\lambda <0
    \end{equation}
    for every $t\in (0,1)$. Since $g, v > 0$, we obtain that $q_t-v \leq 0$ a.e. in $\Omega$ from the equation \eqref{B-14-4}. Since the map $t \mapsto t^{p}$ is convex, we get
    \begin{align}\label{B-14-4'}
        v-q_t = v-(tu^p+(1-t)v^p)^{\frac{1}{p}} \leq v-(tu+(1-t)v) =t(v-u).
    \end{align}
    Thus for all $t \in (0,1)$, $|g v^{p-1}\left(\frac{q_t-v}{t}\right)|\leq gv^{p-1}(v-u)$ which is an integrable function. Also,
    \begin{align}\label{B-14-4''}
        \lim\limits_{t \rightarrow 0}g v^{p-1}\left(\frac{q_t-v}{t}\right)&=g v^{p-1}\lim\limits_{t \rightarrow 0}\left(\frac{q_t-q_0}{t}\right)\nonumber\\
        &=\frac{1}{p}[g v^{p-1} v^{1-p}(u^p-v^p)] =\frac{1}{p}(u^p-v^p)g.
    \end{align}
    pointwise in $\Omega.$ By the dominated convergence theorem, $g v^{p-1}\left(\frac{q_t-v}{t}\right) \rightarrow \frac{1}{p}(u^p-v^p)g$ in $L^1(\Omega)$. Thus, applying the limit as $t \rightarrow0$ in the equation \eqref{B-14-4}, we get
    \begin{equation}\label{B-14-5}
        p\lambda \int_{\Omega}  \frac{1}{p} (u^p-v^p)gdx \leq \lambda_1-\lambda \,,\,\,
    \end{equation}
    i.e.,
    \begin{equation}\label{B-14-6}
        0= \lambda \left( \int_{\Omega} g|u|^p dx- \int_{\Omega} g|v|^p dx \right) \leq \lambda_1-\lambda.
    \end{equation}
    This contradicts our assumption that $\lambda>\lambda_1$. Hence the proof is complete.
\end{proof}
As a consequence of the previous theorem, we have the following result: which gives an upper bound of the measure of subset either $\{x\in\Omega: u(x)>0\}$ or $\{x\in\Omega: u(x)<0\}$ for non-principal eigenfunctions $u$.
\begin{cor}\label{B-lem2}
	Let $u$ be an eigenfunction of \eqref{B} corresponding to $\nu(\Omega)\neq\lambda_1(\Omega)$. Then we have $\nu(\Omega)>\lambda_1(\Omega_{+})$ and $\nu(\Omega)>\lambda_1(\Omega_{-})$, where $\Omega_{+}=\{x\in\Omega: u(x)>0\}$ and $\Omega_{-}=\{x\in\Omega: u(x)<0\}$. In particular,	
	\begin{equation}\label{ll-ineq}
	\nu(\Omega) \geq C(N, p, s)\left|\Omega_{+}\right|^{-\frac{p-ps}{N-ps}} \text { and }\, \nu(\Omega) \geq C(N, p,s) \left|\Omega_{-}\right|^{-\frac{p-ps}{N-ps}}.
	\end{equation}
\end{cor}

\begin{proof}
	Since $\nu(\Omega)\neq\lambda_1(\Omega)$, then by Theorem \ref{B-thm3}, $u$ must change sign in $\Omega$. Choosing $v=u_{+}$ in the weak formulation \eqref{B2}, we obtain
	\begin{align}\label{B-l2-1}
		\nu(\Omega) \int_{\Omega_{+}}g\left|u_{+}\right|^{p} dx &= H_{s,p}(u,u_+)+\int_{\Omega_+}V\left|u_{+}\right|^{p} dx \nonumber \\
      &\geq \int_{\Omega_+}|\nabla u_+|^p dx +\int_{\mathbb{R}^{N}}\int_{\mathbb{R}^{N}} \frac{\left|u_{+}(x)-u_{+}(y)\right|^{p}}{|x-y|^{N+ps}} dxdy \nonumber \\
      &\hspace*{1cm}+2^{\frac{p}{2}} \int_{\mathbb{R}^{N}}\int_{\mathbb{R}^{N}}\frac{\left(u_{+}(y) u_{-}(x)\right)^{\frac{p}{2}}}{|x-y|^{N+ps}} dxdy +\int_{\Omega_+}V|u_+|^p dx.
	\end{align}
	Dividing both sides of the above inequality \eqref{B-l2-1} by $\displaystyle{\int_{\Omega_{+}}g\left|u_{+}(x)\right|^{p} dx}$, and using the definition of $\lambda_1(\Omega_+)$, we have
	\begin{align*}
		\nu(\Omega) &\geq \lambda_{1}\left(\Omega_{+}\right) + 2^{\frac{p}{2}} \cfrac{\int_{\mathbb{R}^{N}}\int_{\mathbb{R}^{N}} \frac{\left(u_{+}(y) u_{-}(x)\right)^{\frac{p}{2}}}{|x-y|^{N+ps}} d x d y}{\int_{\Omega_{+}}g\left|u_{+}(x)\right|^{p} dx}.
	\end{align*}
 This shows that $\nu(\Omega)>\lambda_{1}\left(\Omega_{+}\right)$. 
 Also, making use of the H\"older's inequality twice and the Sobolev inequality \eqref{SEE} with the inequality \eqref{B-l2-1}, we get
 \begin{align}\label{B-l2-2}
     \int_{\Omega_{+}}g\left|u_{+}\right|^{p} dx &\leq \left(\int_{\Omega_{+}}|g|^{\frac{N}{ps}}dx \right)^{\frac{ps}{N}}\left(\int_{\Omega_{+}}\left|u_{+}\right|^{\frac{Np}{N-ps}}dx\right)^{\frac{N-ps}{N}} \nonumber \\
     &\leq \|g\|_{L^{\frac{N}{ps}}(\Omega)}\left(\int_{\Omega_{+}}\left|u_{+}\right|^{\frac{Np}{N-p}}dx\right)^{\frac{N-p}{N}}\left(\int_{\Omega_{+}}dx\right)^{\frac{p-ps}{N-ps}}\nonumber \\
     &=|\Omega_+|^{\frac{p-ps}{N-ps}}\|g\|_{L^{\frac{N}{ps}}(\Omega)} \|u_+\|_{L^{p^*}(\Omega)}^p \nonumber\\
     &\leq C(N,p)\|g\|_{L^{\frac{N}{ps}}(\Omega)}|\Omega_+|^{\frac{p-ps}{N-ps}}\|\nabla u_+\|_{L^{p}(\Omega)}^p \nonumber \\
     &= C(N,p,s)|\Omega_+|^{\frac{p-ps}{N-ps}} \int_{\Omega_+}|\nabla u_+|^p dx\,, \text{ where }C(N,p,s)= C(N,p)\|g\|_{L^{\frac{N}{ps}}(\Omega)}\nonumber \\
     & \leq  C(N,p,s)|\Omega_+|^{\frac{p-ps}{N-ps}} \left( H_{s,p}(u,u_+)+\int_{\Omega_+}V|u_+|^p dx \right) \nonumber \\
     &=C(N,p,s)|\Omega_+|^{\frac{p-ps}{N-ps}} \nu(\Omega) \int_{\Omega_{+}}g\left|u_{+}\right|^{p} dx.
 \end{align}
 Dividing both sides in above inequality \eqref{B-l2-2} by $\displaystyle{\int_{\Omega_{+}}g\left|u_{+}\right|^{p} dx}$, we arrive at
 \begin{equation}
     \nu(\Omega) \geq C(N,p,s)|\Omega_+|^{-\frac{p-ps}{N-ps}}.
 \end{equation}	
  Further, following similar arguments as above by replacing $u_+$ with $u_-$, one can obtain $\nu>\lambda_{1}\left(\Omega_{-}\right)$ and $\nu(\Omega) \geq C(N,p,s)\left|\Omega_{-}\right|^{-\frac{p-ps}{N-ps}}$. This completes the proof.
 \end{proof}
The next result discusses the properties of the principal eigenvalue. 
\begin{theorem}\label{B-thm4}
    The principal eigenvalue $\lambda_1$ is simple and isolated.
\end{theorem}
\begin{proof}
    We shall first show that the principal eigenvalue $\lambda_1$ is simple.  Let $u,v$ be two eigenfunctions associated with $\lambda_1$. Without loss of generality, assume that $u,v>0$ in $\Omega$ and $${\int_{\Omega}g|u|^p dx=\int_{\Omega}g|v|^p dx=1}.$$
    Define $q_{u,v}:=\left(\frac{u^p+v^p}{2}\right)^{\frac{1}{p}}$. Then from \eqref{B8}, \eqref{B11} and \eqref{B13} for the case  $t=\frac{1}{2}$,  we obtain $$\int_{\Omega}g|q_{u,v}|^p dx=1,$$
    \begin{equation}\label{B_16}
        |\nabla q_{u,v}|^p \leq \frac{1}{2} |\nabla u|^p+\frac{1}{2}|\nabla v|^p
    \end{equation}
    and 
    \begin{align}\label{B15}
        H_{s,p}(q_{u,v},q_{u,v})+\int_{\Omega}V|q_{u,v}|^p dx &\leq \frac{1}{2}\left( H_{s,p}(u,u)+\int_{\Omega}V|u|^p dx +H_{s,p}(v,v)+\int_{\Omega}V|v|^p dx \right) \nonumber \\
        &=\frac{1}{2}(\lambda_1+\lambda_1) =\lambda_1=\mathcal{R}.
    \end{align}
    Now by the definition of $\mathcal{R}$, we should have the equality in $\eqref{B15}$, and it can be achieved when the equality holds in \eqref{B_16}, and in \eqref{B12} for $t=1/2$. Since the inequality \eqref{B_16} is obtained from \eqref{B10}, and the property of strict convexity of the map $t \mapsto t^p$ has been used to obtain \eqref{B10}. Therefore we must have $\frac{\nabla u}{u} = \frac{\nabla v}{v}$ a.e. in $\Omega$. This implies that  $\nabla(\frac{u}{v})=0$ a.e. in $\Omega$, and hence there exist a constant $c>0$ such that $u=cv$ a.e. in $\Omega$. This asserts that $\lambda_1$ is simple.
    \par We now prove that the principal eigenvalue $\lambda_1$ is isolated. Let $(\lambda, v)$ be an eigenpair for the problem \eqref{B} such that $\lambda>\lambda_1$ with $\displaystyle{\int_{\Omega}g|v|^p dx =1}$. Then by the Theorem \ref{B-thm3}, we have $v_- \neq 0$. Using $v_- \in X_0^{s,p}(\Omega)$ as a test function in weak formulation \eqref{B2} of the problem \eqref{B} for $(\lambda, v)$, we obtain 
    \begin{equation}\label{B16}
        H_{s,p}(v,v_-)-\int_{\Omega}V|v_-|^p dx= -\lambda \int_{\Omega}g|v_-|^p dx.
    \end{equation}
    Note that 
    \begin{align}
        \nabla v \cdot \nabla (v_-)&=-(\nabla (v_-))^2 \text{ a.e. in } \Omega, \label{B17'}\text{ and }\\
        -(v(x)-v(y))(v_-(x)-v_-(y))&=-\left((v_+(x)-v_+(y)-(v_-(x)-v_-(y)\right)(v_-(x)-v_-(y)) \nonumber \\
        &\geq (v_-(x)-v_-(y))^2\label{B17''}.
    \end{align}
    By the inequalities \eqref{B17'} and \eqref{B17''}, we get,
    \begin{equation}\label{B17}
        -H_{s,p}(v,v_-) \geq \int_{\Omega}|\nabla (v_-)|^p dx + \int_{\mathbb{R}^N}\int_{\mathbb{R}^N} \frac{|v_-(x)-v_-(y)|^p}{|x-y|^{N+ps}}dxdy.
    \end{equation}
    So using the Sobolev inequality \eqref{SEE} and the inequality \eqref{B17} with the fact that $V\geq 0$, it follows that 
    \begin{align}\label{B18}
        \frac{1}{C}\|v_-\|_{L^{p^*}(\Omega)} & \leq  \int_{\Omega}|\nabla (v_-)|^p dx\nonumber\\
        &\leq \int_{\Omega}|\nabla (v_-)|^p dx + \int_{\mathbb{R}^N}\int_{\mathbb{R}^N} \frac{|v_-(x)-v_-(y)|^p}{|x-y|^{N+ps}}dxdy   \nonumber \\
        & \leq -H_{s,p}(v,v_-) \nonumber\\
        & \leq -H_{s,p}(v,v_-)+ \int_{\Omega}V|v_-|^p dx \nonumber\\    &= -\left( H_{s,p}(v,v_-)-\int_{\Omega}V|v_-|^p dx\right).
    \end{align}
    Now, using the inequality \eqref{B16} in \eqref{B18} and then making use of H\"older inequality twice, we get 
    \begin{align}\label{B19}
        \frac{1}{C}\|v_-\|_{L^{p^*}(\Omega)}  &\leq  \lambda \int_{\Omega}g|v_-|^p dx = \lambda \int_{\Omega_-}g|v_-|^p dx \leq  \lambda \|g\|_{L^{\frac{N}{p}}(\Omega_-)} \|v_-\|_{L^{p^*}(\Omega)} \nonumber \\
         &= \lambda \left(\int_{\Omega_-} |g|^{\frac{N}{p}} dx\right)^{\frac{p}{N}} \|v_-\|_{L^{p^*}(\Omega)} \leq \lambda \|g\|_{L^{\frac{N}{ps}}(\Omega)}|\Omega_-|^{\frac{p-ps}{N}}\|v_-\|_{L^{p^*}(\Omega)}.
    \end{align}
 Since $\|v_-\|_{L^{p^*}(\Omega)} \neq 0$, we obtain 
    \begin{equation}\label{B20}
        |\Omega_-| \geq \left( \frac{1}{C \lambda \|g\|_{L^{\frac{N}{ps}}(\Omega)}} \right) ^{\frac{N}{p-ps}}.
    \end{equation}
Following the above arguments for $-v$ in place of the eigenfunction $v$, one can infer that
    \begin{equation}\label{B21}
        |\Omega_+| \geq \left( \frac{1}{C \lambda \|g\|_{L^{\frac{N}{sp}}(\Omega)}} \right) ^{\frac{N}{p-ps}}.
    \end{equation}
On the contrary, let us suppose that $\lambda_1$ is not isolated. Then there exists a sequence of eigenvalues $(\nu_n)$ such that $\nu_n \searrow \lambda_1$. Let $u_n$ be an eigenfunction associated with $\nu_n$ such that ${\int_{\Omega}g|u_n|^p dx=1}$. Define $\Omega_{n_\pm}=\{ x \in \Omega: u_n(x) \gtrless 0\}$. For all $n \in \mathbb{N}$, we have $\nu_1 \geq \nu_n$. Hence by \eqref{B20} we have 
    \begin{equation}\label{B22}
        |\Omega_{n_-}| \geq \left( \frac{1}{C \nu_1 \|g\|_{L^{\frac{N}{sp}}(\Omega)}} \right) ^{\frac{N}{p-ps}} = M \,\,\text{(say)}.
    \end{equation}
Further, since $(\nu_n)$ is a bounded sequence, hence proceeding as in the proof of Theorem \ref{B-thm1}, it follows that, up to a subsequence (still denoted by $(u_n)$), we have
\begin{align*}
    u_n &\rightharpoonup u \,\,\text{in }\, X_0^{s,p}(\Omega),\\
    u_n&\rightarrow u \,\,\text{in }\, L^q(\Omega) \,\, \text{for } \, q\in (1, p^*) \,,\,\text{and}\,\\
    u_n &\rightarrow u \,\,\text{a.e. in } \Omega.
\end{align*}
 Moreover, we have $\displaystyle{\int_{\Omega} g|u|^p dx=1}$ and $u$ is an eigenfunction corresponding to $\lambda_1$. Without loss of generality, let us assume $u>0$ in $\Omega$. As $u_n \rightarrow u$ a.e. in $\Omega$, by Egorov's theorem there exists a compact set $A \subset \Omega$ such that $|\Omega \setminus A| <\frac{M}{4}$ and $u_n \rightarrow u$ uniformly in $A$. Since $u>0$ in $\Omega$ and $A$ is compact, there exists $\epsilon>0$ such that $u\geq \epsilon$ in $A$. Since $u_n \rightarrow u$ uniformly in $A$, there exists $n_0\in \mathbb{N}$ such that for all $x\in A$,
    $$|u_n(x) -u(x)| < \frac{\epsilon}{2} \,,\, \forall\, n \geq n_0. $$
    This implies that
    $$u_{n_0}(x) > u(x)- \frac{\epsilon}{2} \geq \epsilon-\frac{\epsilon}{2}=\frac{\epsilon}{2},\,\, \forall\, x \in A.$$
    As a consequence we have $A \subset \Omega_{{n_0}_+}$. This again implies that  $\Omega_{{n_0}_-} \subset \Omega \setminus A$. Thus we get $|\Omega_{{n_0}_-}|\leq |\Omega \setminus A|<M$, which is a contradiction to \eqref{B22}. Hence the proof is complete.
\end{proof}
\section{A monotone sequence of positive eigenvalues and the boundedness of eigenfunctions}
The present section is devoted to showing the existence of infinitely many positive eigenvalues to \eqref{B}, and it has been obtained by using the {\it Ljusternik-Schnirelmann theory} \cite{S88}. We also show that the set of all positive eigenvalues to \eqref{B} is a {\it closed set}. In addition to this, the {\it boundedness} of all eigenfunctions associated with positive eigenvalues of \eqref{B} has been proved. 
\begin{theorem}\label{B-thm6}
    Let $q\in \left(\frac{N}{sp}, \infty\right)$ and  $V, g \in L^q(\Omega)$ where $V\geq 0$, $g>0$ a.e. in $\Omega$. Then weighted eigenvalue problem \eqref{B} possesses a sequence of positive eigenvalues diverging to $\infty$, i.e., there exists a sequence of eigenvalues $(\lambda_n)$ to the problem \eqref{B} such that
    \begin{equation*}
        0<\lambda_1 \leq \lambda_2 \leq \cdots\leq \lambda_n \leq \cdots
    \end{equation*}
    and
     $\lambda_n \rightarrow \infty$ as $n \rightarrow \infty$.
\end{theorem}
\begin{proof}
    For $u \in X_0^{s,p}(\Omega)$, define
    \begin{align}\label{t-h}
        T(u):=\frac{1}{p}\left [H_{s,p}(u,u)+\int_{\Omega}V|u|^pdx\right] \,\,\text{ and }\,\,
        H(u):=\frac{1}{p}\int_{\Omega}g|u|^pdx.
    \end{align}
    Clearly, $T$ is bounded from below in $X_0^{s,p}(\Omega)$, since $\frac{1}{p}\|u\|_{X_0^{s,p}(\Omega)}^p \leq T(u)$ for all $u \in X_0^{s,p}(\Omega)$. Moreover, $T$ is even and  $T(0)=0$.\\
    \noindent \textbf{Claim:} $T$ satisfies {\it Palais-Smale condition} on $\mathcal{M}=\{u \in X_0^{s,p}(\Omega) : pH(u)=1\}$.\\
    To prove this, consider a sequence $(u_n)$ in $\mathcal{M}$ satisfying 
    \begin{enumerate}[(i)]
       \item $|T(u_n)| \leq M$ for all $n \in \mathbb{N}$.
        \item For $t_n = \frac{\langle T'(u_n), u_n\rangle}{\langle H'(u_n), u_n\rangle}$, we have $T'(u_n)- t_n H'(u_n) \rightarrow 0$ as $n \rightarrow \infty$, that is,  
            \begin{equation*}
            \langle T'(u_n)-t_n H'(u_n), w \rangle \rightarrow 0 \text{ as } n \rightarrow \infty,~ \forall\, w \in X_0^{s,p}(\Omega).
            \end{equation*}
    \end{enumerate}
    We guarantee the existence of a convergent subsequence of $(u_n)$ in $\mathcal{M}$.
    By (i), we have  $$0 \leq \|u_n\|_{X_0^{s,p}(\Omega)}^p \leq pT(u_n) \leq Mp~\text{for every}~n\in \mathbb{N}.$$
    This implies that the sequence $(u_n)$ is bounded in $X_0^{s,p}(\Omega)$. Since $X_0^{s,p}(\Omega)$ is a reflexive Banach space, we have $u_n\rightharpoonup u$ weakly in $X_0^{s,p}(\Omega)$. By the compact embedding in \eqref{CPT MET}, we get $u_n \rightarrow u$ (upto a subsequence) in $L^q(\Omega)$ for all $1 \leq q <p^*$ and $u_n\rightarrow u$ a.e. in $\Omega$. Therefore, proceeding as in the proof of Theorem \ref{B-thm1}, we get 
    \begin{equation}\label{B23}
        \int_{\Omega}g|u|^p dx=1.
    \end{equation}
    This shows that $u \not\equiv 0$ and $u \in \mathcal{M}$. By the definition of $t_n$, we have
    \begin{equation}\label{B24}
        \langle T'(u_n)-t_n H'(u_n), u_n \rangle =0
        \text{ for all } n\in \mathbb{N}.   
    \end{equation}
    Also, using condition (ii), we obtain
    \begin{equation}\label{B25}
        \langle T'(u_n)-t_n H'(u_n), u \rangle \rightarrow 0
        \text{ as } n\rightarrow \infty \,\,\forall\, u \in X_0^{s,p}(\Omega).
    \end{equation}
    Thus \eqref{B24} we get
    \begin{equation}\label{B26}
        \langle T'(u_n)-t_n H'(u_n), u_n-u \rangle \rightarrow 0 \text{ as } n\rightarrow \infty.    
    \end{equation}
    On using the condition (i) and $u_n\in \mathcal{M}$, we have
    \begin{align}\label{B27}
        \langle t_n H'(u_n), u_n-u \rangle &= \frac{\langle T'(u_n), u_n\rangle}{\langle H'(u_n), u_n\rangle}\langle H'(u_n), u_n-u \rangle \nonumber \\
        &= \frac{\langle T'(u_n), u_n\rangle}{pH(u_n)} \langle  H'(u_n), u_n-u \rangle \nonumber \\
        &= \langle T'(u_n), u_n\rangle \int_{\Omega} g|u_n|^{p-2} u_n (u_n-u) dx \nonumber \\
        & = p T(u_n) \int_{\Omega} g|u_n|^{p-2} u_n (u_n-u) dx \nonumber \\
        &\leq Mp \int_{\Omega} g|u_n|^{p-2} u_n (u_n-u) dx.
    \end{align}
    Again, using $g\in L^{\frac{N}{sp}}(\Omega)$ and the H\"older's inequality, we get
    \begin{equation}\label{B28}
        \int_{\Omega} g|u_n|^{p-2} u_n (u_n-u) dx \leq \|g\|_{L^{\frac{N}{ps}}(\Omega)} \left(\int_{\Omega} ||u_n|^{p-2}u_n(u_n-u)|^{\frac{N}{N-sp}}dx\right)^{\frac{N-sp}{N}}.
    \end{equation}
    Since, $u_n\rightarrow u$ a.e. in $\Omega$, we have $|u_n|^{p-2} u_n (u_n-u) \rightarrow 0$ a.e. in $\Omega$. Moreover, we have $p_s^* <p^*$. Thus by \cite[Theorem 4.9]{B11} there exists $h \in L^{p_s^*}(\Omega)$ such that $|u_n| \leq h$ a.e. in $\Omega$, which implies $||u_n|^{p-2} u_n (u_n-u)|^{\frac{N}{N-sp}} \leq 2^\frac{N}{N-sp}h^{p_s^*}$. Therefore, using the dominated convergence theorem on the RHS of the integral \eqref{B28}, we obtain
    \begin{equation}\label{B29}
        \int_{\Omega} g|u_n|^{p-2} u_n (u_n-u) dx \rightarrow 0 \text{ as } n \rightarrow \infty.
    \end{equation}
    Following the arguments as above, we have 
    \begin{equation}\label{B30}
        \int_{\Omega} V|u_n|^{p-2} u_n (u_n-u) dx \rightarrow 0 \text{ as } n \rightarrow \infty.
    \end{equation}
    Since $V\geq 0$ and $g>0$ in $\Omega$, from inequality \eqref{B27} we derive that 
    \begin{equation}\label{B30'}
         \langle t_n H'(u_n), u_n-u \rangle \rightarrow 0 \text{ as } n \rightarrow \infty.
    \end{equation}
    Using \eqref{B30'} in \eqref{B26}, it follows that 
    \begin{equation*}
        \langle T'(u_n), u_n-u \rangle \rightarrow 0 \text{ as } n\rightarrow \infty. 
    \end{equation*}
    Moreover, as $u_n \rightharpoonup u$ weakly in $X_0^{s,p}(\Omega)$ implies
    \begin{equation*}\label{B32'}
         \langle T'(u), u_n-u \rangle \rightarrow 0 \text{ as } n\rightarrow \infty.
    \end{equation*}
    Consequently, we get
    \begin{equation}\label{B31'}
         \langle T'(u_n)-T'(u), u_n-u \rangle \rightarrow 0 \text{ as } n\rightarrow \infty. 
    \end{equation}
    Now, we set 
    $$\langle T'(u_n)-T'(u), u_n-u \rangle =J_1+J_2+J_3,$$
    where
    \begin{equation*}\label{Bi1}
        J_1=\int_{\Omega}(| \nabla u_n | ^{p-2}\nabla u_n -| \nabla u| ^{p-2}\nabla u)(\nabla u_n-\nabla u) dx,
    \end{equation*}
    \begin{align*}\label{Bi2}
        J_2 =& \int_{\mathbb{R}^N}\int_{\mathbb{R}^N} \left[\frac{|u_n(x)-u_n(y)| ^{p-2}(u_n(x)-u_n(y))(u_n(x)-u_n(y)-(u(x)-u(y))}{|x-y|^{N+sp}}\right] dxdy\nonumber\\
        &-\int_{\mathbb{R}^N}\int_{\mathbb{R}^N}\left[\frac{| u(x)-u(y)| ^{p-2}(u(x)-u(y))(u_n(x)-u_n(y)-(u(x)-u(y))}{|x-y|^{N+sp}}\right] dx dy,
    \end{align*}
    and 
    \begin{equation*}\label{Bi3}
        J_3=\int_{\Omega} V(|u_n|^{p-2} u_n- |u|^{p-2} u )  (u_n-u) dx.
    \end{equation*}
    Since $u_n \rightharpoonup u$ weakly in $X_0^{s,p}(\Omega)$, we have
    \begin{equation}\label{B31}
         \int_{\Omega} V|u|^{p-2} u(u_n-u) dx \rightarrow 0 \text{ as } n \rightarrow \infty.
    \end{equation}
    As an outcome of \eqref{B30} and \eqref{B31}, it follows that $J_3\rightarrow 0$ as $n\rightarrow \infty$.
    Thus, we get
    \begin{equation}\label{B33}
        J_1+J_2 \rightarrow 0 \text{ as } n \rightarrow \infty.
    \end{equation}
       Recall the following Simon's inequalities
     \begin{align}
        &(| \xi| ^{p-2}\xi-| \eta| ^{p-2}\eta)(\xi-\eta) \geq A | \xi-\eta| ^p, \ p \geq2\label{ts1} \text{ and }\\
        &(|\xi| ^{p-2}\xi-|\eta| ^{p-2}\eta)(\xi-\eta) \geq A\frac{| \xi-\eta| ^2}{(|\xi| +| \eta| )^{2-p}}, \ 1<p<2.\label{ts2}
    \end{align}
    where $\xi,~\eta \in \mathbb{R}$ and $A>0$ is a constant. For the case $p \geq 2$, using the equation \eqref{ts1} with $\xi=\nabla u_n, ~\eta=\nabla u$ in $J_1$, we get
    \begin{equation}\label{B35'}
        \int_{\Omega} |\nabla(u_n-u)|^p dx \leq A J_1.
    \end{equation}
    Similarly, using \eqref{ts1} with $\xi=u_n(x)-u_n(y), ~\eta=u(x)-u(y)$ in $J_2$, we have
    \begin{equation}\label{B35''}
        \int_{\mathbb{R}^N}\int_{\mathbb{R}^N} \frac{|u_n(x)-u_n(y)-(u(x)-u(y))|^p}{|x-y|^{N+sp}} dx dy \leq A J_2.
    \end{equation}
    On combining \eqref{B35'} and \eqref{B35''}, we obtain 
    \begin{equation}\label{B35-1}
        \|u_n-u\|_{X_0^{s,p}(\Omega)}^p \leq A(J_1+J_2).
    \end{equation}
    Again, for the case $1<p<2$, taking $\xi=\nabla u_n, ~\eta=\nabla u$ in \eqref{ts2} and substituting in $J_1$, we get
    \begin{equation}\label{B35'''}
        \int_{\Omega} |\nabla(u_n-u)|^p dx \leq A \int_{\Omega} (| \nabla u_n| ^{p-2}\nabla u_n-| \nabla u| ^{p-2}\nabla u)^{\frac{p}{2}}(\nabla u_n-\nabla u)^{\frac{p}{2}}(| \nabla u_n| +| \nabla u| )^{\frac{p(2-p)}{2}} dx.
    \end{equation}
    Since $\displaystyle{\int_{\Omega} |\nabla v|^p dx} \leq \|v\|_{X_0^{s,p}(\Omega)}^p$ for all $v \in X_0^{s,p}(\Omega)$, by the H\"older's inequality with the exponents $\frac{2}{p}$ and $\frac{2}{2-p}$ in \eqref{B35'''}, we get
    \begin{align}\label{B35'''-1}
        \int_{\Omega} |\nabla(u_n-u)|^p dx &\leq A \left(\int_{\Omega} (| \nabla u_n| ^{p-2}\nabla u_n-| \nabla u| ^{p-2}\nabla u)(\nabla u_n-\nabla u) dx\right)^{\frac{p}{2}} \nonumber\\
        & \hspace{2cm}\times \left( \int_{\Omega}(| \nabla u_n| +| \nabla u| )^p dx\right)^{\frac{2-p}{2}}\nonumber\\
        &\leq A(2M)^\frac{(2-p)}{2} \left(\int_{\Omega} (| \nabla u_n| ^{p-2}-| \nabla u| ^{p-2})(\nabla u_n-\nabla u) dx\right)^{\frac{p}{2}} \nonumber\\
        &= C J_1^{\frac{p}{2}}, \text{ where } C=A(2M)^\frac{(2-p)}{2} .
    \end{align}
    Proceeding as above with $\xi=u_n(x)-u_n(y), ~\eta=u(x)-u(y)$ in \eqref{ts2}, we deduce
    \begin{align}\label{B35'''-2}
        \int_{\mathbb{R}^N}\int_{\mathbb{R}^N}\frac{|(u_n-u)(x)-(u_n-u)(y)|^p}{|x-y|^{N+sp}}dxdy \leq C J_2^{\frac{p}{2}}. 
    \end{align}
    Now, combining \eqref{B35'''-1}, \eqref{B35'''-2} and using the convexity of the map $t \mapsto t^\frac{2}{p}$, we get 
    \begin{equation}\label{B-35-2}
        \|u_n-u\|_{X_0^{s,p}(\Omega)}^2 \leq C (J_1+J_2).
    \end{equation}
    Therefore, using \eqref{B33}, \eqref{B35-1} and \eqref{B-35-2}, we obtain
    \begin{equation}\label{B34}
        \|u_n-u\|_{X_0^{s,p}(\Omega)}^p = H_{s,p}(u_n-u,u_n-u) \rightarrow 0 \text{ as } n\rightarrow \infty.
    \end{equation}
    Hence $T$ satisfies the Palais-Smale condition.
      \par Therefore, applying the \textit{Ljusternik-Schnirelmann principle} \cite{S88} on $T$, it follows that the problem \eqref{B} has a sequence of positive critical values
    \begin{equation*}\label{B35}
        0<\lambda_1 \leq \lambda_2 \leq \cdots \leq \lambda_n \leq\cdots,
    \end{equation*}
    defined  by
    $$ \displaystyle{\lambda_n= \inf_{A \in \Sigma_n} \max_{u \in A} pT(u) =\inf_{A \in \Sigma_n} \max_{u \in A} \left(H_{s,p}(u,u)+\int_{\Omega}V|u|^p dx \right)},$$
    where $\Sigma_n= \{ A \in \Sigma \mid \gamma(A) \geq n \}$ and $\Sigma=\{A \subset X_0^{s,p}(\Omega) \setminus \{0\} \mid A \text{ is compact and }A=-A \}$. We now show that $\lambda_n \rightarrow \infty$ as $n \rightarrow \infty$. We prove it by contradiction. Suppose there exists $L>0$ such that $0<\lambda_n \leq L$ for all $n \in \mathbb{N}$. Since, $X_0^{s,p}(\Omega)$ is separable, hence $X_0^{s,p}(\Omega)$ admits a biorthogonal system $\{w_n, w_n^*\}$ with the following properties. $X_0^{s,p}(\Omega)=\overline{\text{span}\{w_n:n\in\mathbb{N}\}}$ such that for all $w_n^*\in (X_0^{s,p}(\Omega))^*$ we have $\langle w_i^*,w_j \rangle= \delta_{i,j}$. Moreover, $\langle w_n^*, v\rangle =0 \,\,\forall \,n\in\mathbb{N}$ implies that $v=0$, (see \cite{P76}).\\
    Let
    \begin{equation*}
        X_n=\overline{\mbox{span}\{w_{n}, w_{n+1}, \cdots\}}~\,\,\,\text{and}~\,\,\,a_n=\inf_{A \in \Sigma_n} \sup_{u \in A\cap X_n} pT(u).
    \end{equation*}
    Note that the co-dimension of $X_n$ is $n-1$. By Theorem \ref{genus}-$(vi)$, we have $A \cap X_n \neq \emptyset$ for all $A \in \Sigma_n$.  This shows that $\sup\limits_{u\in A\cap X_n}pT(u)>0$. Furthermore, $a_n \leq \lambda_n \leq L,~\forall\, n \in \mathbb{N}$. Now, for each $n\in \mathbb{N}$, choose $v_n \in A \cap X_n$ such that
    \begin{equation*}
        \int_{\Omega}g|v_n|^p dx=1~\text{ and }~0\leq a_n \leq pT(v_n) \leq L+1,\ n \in \mathbb{N}.
    \end{equation*}
    This implies that $(v_n)$ is a bounded in $X_0^{s,p}(\Omega)$. Therefore, proceeding as in the proof of Theorem \ref{B-thm1}, we ensure the existence of an element $v \in X_0^{s,p}(\Omega)$ such that $v_n \rightharpoonup v$ in $X_0^{s,p}(\Omega)$ with ${\int_{\Omega}g|v|^p dx=1}$. Therefore $v \not\equiv 0$ in $\Omega$. However, by the choice of the biorthogonal system, we have
    \begin{equation*}\label{B36}
        \langle w_m^*, v \rangle =\lim\limits_{n \rightarrow \infty} \langle w_m^*, v_n\rangle=0, ~\text{for every}~ m\in \mathbb{N},
    \end{equation*}
    which implies $v=0$. This gives a contradiction. Hence, $\lambda_n\rightarrow \infty$ as $n\rightarrow\infty$. This completes the proof.
\end{proof}
\begin{theorem}\label{eigen closed}
    The set of all positive eigenvalues to \eqref{B} is closed.
\end{theorem}
\begin{proof}
    Let $(\nu_n)$ be a sequence of eigenvalues of the problem $\eqref{B}$ such that $\nu_n \rightarrow \nu$. Then, $(\nu_n)$ is a bounded sequence. For each $n \in \mathbb{N}$, let $u_n$ be an eigenfunction corresponding to the eigenvalue $\nu_n$ such that $\int_\Omega g|u_n|^p dx=1$. Then by taking $u_n$ as the test function for the eigenpair $(\nu_n,u_n)$ in the weak formulation \eqref{B2} and using the fact that $V \geq 0$, we have
    \begin{equation*}
        \|u_n\|_{X_0^{s,p}(\Omega)}^p \leq H_{s,p}(u_n,u_n) + \int_{\Omega}V|u_n|^pdx=\nu_n.
    \end{equation*}
    Therefore $(u_n)$ is a bounded sequence in $X_0^{s,p}(\Omega)$. Since $X_0^{s,p}(\Omega)$ is a reflexive Banach space, $u_n \rightharpoonup u$ weakly in $X_0^{s,p}(\Omega)$. Then, by the compact embedding \eqref{CPT MET}, $u_n \rightarrow u$ up to a subsequence in $L^q(\Omega)$ for $1 \leq q <p^*$. Recall the functional $T$ as defined in \eqref{t-h}. Then for any $v,w \in X_0^{s,p}(\Omega)$, we have
    \begin{equation}\label{B-EC-1}
        \langle T'(v), w\rangle=H_{s,p}(v,w)+\int_\Omega V|v|^{p-2}vw dx.
    \end{equation}
    Since $u_n \rightharpoonup u$ in $X_0^{s,p}(\Omega)$, we get
    \begin{equation}\label{B-EC-2}
        \langle T'(u), u_n-u \rangle \rightarrow 0 \,\, \text{ as } n \rightarrow \infty.
    \end{equation}
    As $u_n$ is an eigenfunction associated with $\nu_n$, we deduce
    \begin{equation}\label{B-EC-3}
        \langle T'(u_n), u_n-u \rangle= \nu_n \int_\Omega g|u_n|^{p-2}u_n(u_n-u) dx
    \end{equation}
    and
    \begin{align}\label{B-EC-5}
        \int_\Omega g|u_n|^{p-2}u_n (u_n-u) dx \rightarrow 0 \,\, \text{ as } n\rightarrow \infty.
    \end{align}
    Therefore from \eqref{B-EC-3}, \eqref{B-EC-5} and using the fact that $(\nu_n)$ is bounded, we obtain
    \begin{equation}\label{B-EC-6}
        \langle T'(u_n), u_n-u \rangle \rightarrow 0 \,\, \text{ as } n\rightarrow \infty.
    \end{equation}
     Again from \eqref{B-EC-2} and \eqref{B-EC-6}, we have
     \begin{equation}\label{B-EC-7}
          \langle T'(u_n)-T'(u), u_n-u \rangle \rightarrow 0 \,\, \text{ as } n\rightarrow \infty.
     \end{equation}
     Following similar arguments as in Theorem \ref{B-thm6}, we obtain $u_n \rightarrow u$ in $X_0^{s,p}(\Omega)$. Therefore, we deduce that
     \begin{align}
         \nabla u_n &\rightarrow \nabla u \,\, \text{ in } L^p(\Omega) \text{ and } \label{B-EC-8}\\
         \frac{u_n(x)-u_n(y)}{|x-y|^{\frac{N}{p}+s}} &\rightarrow \frac{u(x)-u(y)}{|x-y|^{\frac{N}{p}+s}} \,\, \text{ in } L^p(\mathbb{R}^N) \times L^p(\mathbb{R}^N) \label{B-EC-9}.
     \end{align}
     By \cite[Theorem 4.9]{B11}, there exists $h \in L^p(\Omega)$ such that $|\nabla u_n| \leq h$ a.e. in $\Omega$ for every $n \in \mathbb{N}$. This implies
     \begin{align*}
         \left||\nabla u_n|^{p-2}\nabla u_n -|\nabla u|^{p-2}\nabla u \right|  &\rightarrow  0 \text{  a.e. in } \Omega \text{ and }\\
         \left||\nabla u_n|^{p-2}\nabla u_n -|\nabla u|^{p-2}\nabla u \right|^{\frac{p}{p-1}} &\leq 2^{\frac{p}{p-1}} h^p \text{ a.e. in } \Omega.
     \end{align*} 
     Thus by the dominated convergence theorem, we have
     \begin{equation}\label{B-EC-10}
         A_n=\int_\Omega \left||\nabla u_n|^{p-2}\nabla u_n -|\nabla u|^{p-2}\nabla u \right|^{\frac{p}{p-1}} dx \rightarrow 0
          \,\, \text{ as } n\rightarrow \infty.
     \end{equation}
     For $v \in X_0^{s,p}(\Omega)$, define
     \begin{equation*}
         f_v(x,y)= \frac{|v(x)-v(y)|^{p-2}(v(x)-v(y))}{|x-y|^{\frac{(N+ps)(p-1)}{p}}}.
     \end{equation*}
     On the other hand, using \eqref{B-EC-9} and following the steps as in \eqref{B-EC-10}, one can obtain
     \begin{equation}\label{B-EC-11}
         B_n=\int_{\mathbb{R}^N}\int_{\mathbb{R}^N}  \left| f_{u_n}(x,y) -  f_{u}(x,y) \right|^{\frac{p}{p-1}} dy dx \rightarrow 0 \,\, \text{ as } n \rightarrow \infty.
     \end{equation}
      Since $p_s^*<p^*$, we have $u_n \rightarrow u$ in $L^{p_s^*}(\Omega)$. Thus by \cite[Theorem 4.9]{B11}, there exists $h \in L^{p_s^*}(\Omega)$ such that $|u_n| \leq h$ a.e. in $\Omega$. Therefore, we have $$\left||u_n|^{p-2}u_n-|u|^{p-2}u\right|^{\frac{Np}{(N-ps)(p-1)}}\leq 2^{\frac{Np}{(N-ps)(p-1)}}h^{p_s^*}.$$
      Again by the dominated convergence theorem, we get
      \begin{equation}\label{B-EC-12}
          C_n=\int_\Omega \left||u_n|^{p-2}u_n-|u|^{p-2}u\right|^{\frac{Np}{(N-ps)(p-1)}} dx \rightarrow 0 \,\, \text{ as } n \rightarrow \infty.
      \end{equation}
      Since, $v \in X_0^{s,p}(\Omega)$, applying the H\"older's inequality on the first term of $H_{s,p}$ with the exponents $p$ and $\frac{p}{p-1}$ and then using \eqref{B-EC-10}, we arrive
      \begin{align}\label{B-EC-13}
          \int_\Omega \left(|\nabla u_n|^{p-2}\nabla u_n-|\nabla u|^{p-2}\nabla u \right) v dx &\leq A_n^{\frac{p-1}{p}}\|v\|_{L^p(\Omega)} \rightarrow 0
          \,\, \text{ as } n\rightarrow \infty.
      \end{align}
       Further, on applying the H\"older's inequality twice on the second term of $H_{s,p}$ and using \eqref{B-EC-11}, we get
      \begin{align}\label{B-EC-14}
          \int_{\mathbb{R}^N}\int_{\mathbb{R}^N}\left( f_{u_n}(x,y)-f_u(x,y)\right)(v(x)-v(y)) dy dx \rightarrow 0 \,\, \text{ as } n\rightarrow \infty.
      \end{align}
      Thus from \eqref{B-EC-13} and \eqref{B-EC-14}, we have
      \begin{equation}\label{B-EC-18}
          H_{s,p}(u_n,v) \rightarrow  H_{s,p}(u,v) \,\, \text{ as } n\rightarrow \infty.
      \end{equation}
      Again applying the H\"older's inequality twice on the last integral of \eqref{B-EC-1} and using \eqref{B-EC-12}, we obtain
      \begin{align}\label{B-EC-15}
          \int_\Omega V\left( |u_n|^{p-2}u_n-|u|^{p-2}u \right)v dx &\leq \|V\|_{L^{\frac{N}{ps}}(\Omega)} \left(\int_\Omega \left|\left( |u_n|^{p-2}u_n-|u|^{p-2}u \right)v\right|^{\frac{N}{N-sp}} dx \right)^{\frac{N-ps}{N}} \nonumber\\
          &\leq \|V\|_{L^{\frac{N}{ps}}(\Omega)} \left( C_n^{\frac{p-1}{p}} \|v\|_{L^{p_s^*}(\Omega)}^{\frac{p_s^*}{p}}\right)^{\frac{N-ps}{N}} \nonumber\\
          &\rightarrow 0 \,\, \text{ as } n\rightarrow \infty.
      \end{align}
      Similarly, we derive
      \begin{equation}\label{B-EC-16}
          \int_\Omega g\left( |u_n|^{p-2}u_n-|u|^{p-2}u \right)v dx \rightarrow 0 \,\, \text{ as } n\rightarrow \infty.
      \end{equation}
      Since $u_n$ is an eigenfunction associated with $\nu_n$, we have
      \begin{equation}\label{B-EC-17}
          H_{s,p}(u_n,v)+\int_\Omega V|u_n|^{p-2}u_n v dx =\nu_n \int_\Omega g|u_n|^{p-2}u_n v dx, \,\, \text{ for all }   v\in X_0^{s,p}(\Omega).
      \end{equation}
      Since $\nu_n \rightarrow \nu$ as $n \rightarrow \infty$, hence, passing to the limit as $n \rightarrow \infty$ in \eqref{B-EC-17} and using \eqref{B-EC-18}, \eqref{B-EC-15} and \eqref{B-EC-16}, we get
      \begin{equation}\label{B-EC-19}
          H_{s,p}(u,v)+\int_\Omega V|u|^{p-2}u v dx =\nu_n \int_\Omega g|u|^{p-2}u v dx. 
      \end{equation}
      This shows that $(\nu,u)$ is an eigenpair to the problem \eqref{B}. Hence the result.
\end{proof}
\noindent Finally, the following theorem asserts that eigenfunctions associated with positive eigenvalue $\lambda$ of \eqref{B} are bounded. 
\begin{theorem}\label{B-thm5}
    Let $u$ be an eigenfunction corresponding to an eigenvalue $\lambda>0$ of the problem \eqref{B}. Then, $u \in L^{\infty}(\mathbb{R}^N)$.
\end{theorem}
\begin{proof}
    For each $n\in \mathbb{N}$, define $$u_n:=(u-(1-2^{-n}))_+;~t_n:=\|u_n\|_{L^{p_s^*}(\Omega)} \text{ and } v_n=u-(1-2^{-(n+1)}).$$
    Therefore, $\forall\, n\in\mathbb{N}$, we have $u_n \in X_0^{s,p}(\Omega)$ and $0 \leq u_{n+1}\leq u_n$ and hence $t_{n+1} \leq t_n$. Moreover, $(v_n)_+= u_{n+1}$ and $\nabla v_n=\nabla u$. Now, we have 
    $$\nabla (v_n)_+=0 \,\,\text{a.e. in}\,\, \{x\in\Omega: v_n \leq 0\}$$ and 
    $$ |\nabla u_{n+1}|^2= |\nabla (v_n)_+|^2= \nabla (v_n)_+\cdot \nabla (v_n)_+ = \nabla (v_n)_+\cdot \nabla u= \nabla u_{n+1}\cdot \nabla u, $$
     a.e. in $\{x\in\Omega: v_n(x)>0\}$. Note that, the following inequality holds true for any $v\in X_0^{s,p}(\Omega)$.
     \begin{equation}\label{B_17}
        |v(x)-v(y)|^{p-2}(v(x)-v(y))(v_+(x)-v_+(y)) \geq |v_+(x)-v_+(y)|^p,\,\,\, \forall\,x,y \in \mathbb{R}^N.
    \end{equation}
    Also, $u_{n+1}(x) \neq 0$,implies that $u(x)>1-2^{-(n+1)}>0$. Therefore, using the inequality \eqref{B_17}, we deduce that
    \begin{align}\label{B37}
        \|u_{n+1}\|_{X_0^{s,p}(\Omega)} &\leq H_{s,p}(u,u_{n+1}) \nonumber \\
        &\leq H_{s,p}(u,u_{n+1}) + \int_{\Omega}V|u|^{p-2}u u_{n+1}dx\nonumber \\
        &= \lambda \int_{\{u_{n+1}>0\}} g|u|^{p-2}u u_{n+1} dx. 
    \end{align}
    Now, for every $x\in A:=\{x\in \Omega: u_{n+1}(x)>0\}$, we have
    \begin{align*}
        (2^{n+1}-1)u_{n}(x)-u(x)&=(2^{n+1}-1)u(x)-\frac{(2^{n+1}-1)((2^{n}-1))}{2^n}-u(x) \\
        &=2(2^n-1)u(x)-\frac{(2^{n+1}-1)((2^{n}-1))}{2^n} \\
        &=2(2^n-1)(u(x)-\frac{2^{n+1}-1}{2^{n+1}})\\
        &=2(2^n-1)u_{n+1}(x)\\
        &>0.
    \end{align*}
    Hence, $0<u(x)<(2^{n+1}-1)u_{n}(x)$. Therefore, using this estimate, \eqref{B37} together with the H\"older's inequality, we obtain
    \begin{align}\label{B38}
        \|u_{n+1}\|_{X_0^{s,p}(\Omega)}^p &\leq (2^{n+1}-1)^{p-1}\lambda \|g\|_{L^{\frac{N}{ps}}(\Omega)} \|u_{n}\|_{L^{p_s^*}(\Omega)}^{\frac{p}{p_s^*}} \nonumber\\
        &\leq 2^{(n+1)p}\lambda t_n^{\frac{p}{p_s^*}} \|g\|_{L^{\frac{N}{ps}}(\Omega)}.
    \end{align}
    Since $\{x\in\Omega: u_{n+1}(x)>0\}  \subset \{x\in\Omega: u_n(x)>2^{-(n+1)}\}$, we get
    \begin{align}\label{B39}
        |\{u_{n+1}>0\} | &\leq |\{u_n>2^{-(n+1)}\} | \nonumber\\
        & =\int_{\{u_n>2^{-(n+1)}\}}| u_n| ^{p_s^*}| u_n| ^{-p_s^*} dx \nonumber \\ & \leq  2^{(n+1)p_s^*} \int_{\{u_n>2^{-(n+1)}\}}| u_n| ^{p_s^*} dx \nonumber\\
        &=2^{(n+1)p_s^*} t_n^{p_s^*}.
    \end{align}
    Thus using the H\"older's inequality with exponent $\frac{N-sp}{N-p}>1$ and $\frac{N-ps}{p-sp}$, and then applying the mixed Sobolev inequality \eqref{MSI} together with inequalities \eqref{B38} and \eqref{B39}, we get
    \begin{align}\label{B40'}
        t_{n+1}^{p_s^*} &= \int_{\{u_{n+1}>0\}}|u_{n+1}|^{p_s^*} dx
        \leq |\{u_{n+1}>0\} |^{\frac{p-sp}{N-sp}} \|u_{n+1}\|_{L^{p^*}(\Omega)}^{p^*\frac{N-p}{N-sp}} \nonumber \\
        & \leq  2^{(n+1)p_s^*\frac{p-sp}{N-sp}} C t_n^{p_s^*\frac{p-sp}{N-sp}}\|u_{n+1}\|_{X_0^{s,p}(\Omega)}^{p^*\frac{N-p}{N-sp}} \nonumber\\
        &=2^{(n+1)p_s^*\frac{p-sp}{N-sp}} C t_n^{p_s^*\frac{p-sp}{N-sp}}\|u_{n+1}\|_{X_0^{s,p}(\Omega)}^{p_s^*} \nonumber \\
        & \leq 2^{(n+1)p_s^*\frac{p-sp}{N-sp}} t_n^{p_s^*\frac{p-sp}{N-sp}} C\left( \lambda \|g\|_{L^{\frac{N}{ps}} (\Omega)} 2^{(n+1)p} t_n^{\frac{p}{p_s^*}} \right)^{\frac{p_s^*}{p}}.
    \end{align}
    From the equation \eqref{B40'}, we get
    \begin{equation}\label{B40}
        t_{n+1} \leq M^n t_n^{1+\beta},
    \end{equation}
    where $\beta =\frac{p-sp}{N-sp}$ and $\max\left\{1,\left(C\lambda^{\frac{p_s^*}{p}} \|g\|_{L^{\frac{N}{ps}} (\Omega)}^{\frac{p_s^*}{p}}2^{p_s^*+p_s^*\frac{p-sp}{N-sp}}\right)^\frac{2}{p_s^*} \right\}<M<\infty$.\\
    Now define, $t_0=\|u_+\|_{L^{p_s^*}(\Omega)}$. Since a nonzero scalar multiple of an eigenfunction is again an eigenfunction to \eqref{B} for the eigenvalue $\lambda$, we can choose $u$ such that $t_0<M^{-\frac{1}{\beta^2}}$. Therefore, from \eqref{B40} we have $t_n \rightarrow 0$ as $n \rightarrow \infty$. However, $t_n \rightarrow \|(u-1)_+\|_{L^{p_s^*}(\Omega)}^{p_s^*}$ as $n \rightarrow \infty$. This implies that $\|(u-1)_+\|_{L^{p_s^*}(\Omega)}=0$ and hence $|u_+| \leq 1$ a.e. in $\Omega$. Since $u=0$ in $\mathbb{R}^N\setminus\Omega$, we get $\|u_+\|_{L^\infty{(\mathbb{R}^N})} \leq 1$. Following the above arguments for $(-u)$, one can obtain that $\|u_-\|_{L^\infty{(\mathbb{R}^N})} \leq 1$. Hence the result follows.
\end{proof}

\section*{Acknowledgement}
The author Lakshmi R. thanks the financial support provided by the Ministry of Education (formerly known as  MHRD), Government of India. RKG acknowledges the DST-FIST program (Govt. of India) for providing financial support for setting up the research and computing lab facility at the Department of Mathematics, The LNM Institute of Technology, Jaipur under the scheme “Fund for Improvement of Science and Technology” (FIST - No. SR/FST/MS-I/2018/24). SG acknowledges the research facilities available at the Department of Mathematics, NIT Calicut under the DST-FIST support, Govt. of India [Project no. SR/FST/MS-1/2019/40 Dated. 07.01 2020.].

\section*{Conflict of interest statement}
On behalf of all authors, the corresponding author states that there is no conflict of interest.

\section*{Data availability statement}
Data sharing does not apply to this article as no datasets were generated or analysed during the current study.

\end{document}